\documentclass[11pt]{amsart}

\usepackage[mathscr]{eucal}
\usepackage{enumerate}
\usepackage{mathrsfs}
\usepackage{stmaryrd}
\usepackage{comment}
\usepackage{enumerate}
\usepackage{hyperref}

\usepackage{amsmath,amsfonts,amssymb}
\usepackage[all]{xy}
\usepackage{pb-diagram}
\usepackage{hyperref}

\usepackage{mathtools}

\usepackage{mathrsfs}

\theoremstyle{plain}
        \newtheorem{theorem}{Theorem}[section]
        \newtheorem{lemma}[theorem]{Lemma}
        \newtheorem{proposition}[theorem]{Proposition}
        \newtheorem{corollary}[theorem]{Corollary}
        
        \theoremstyle{definition}
        \newtheorem{definition}[theorem]{Definition}
        \newtheorem{remark}[theorem]{Remark}

\newcommand{\N}{{\mathbb N}}
\newcommand{\Z}{{\mathbb Z}}

\newcommand{\C}{{\mathbb C}}
\newcommand{\R}{{\mathbb R}}
\newcommand{\mbA}{{\mathbb A}}

\newcommand{\g}{\mathfrak{g}}

\newcommand{\sfG}{\mathsf{G}}
\newcommand{\sfA}{\mathsf{A}}

\newcommand{\Hom}{\operatorname{Hom}}
\newcommand{\symb}{\operatorname{Sym}}
\newcommand{\sym}{\operatorname{Sym}}

\newcommand{\asymb}{\operatorname{ASym}}
\newcommand{\jsymb}{\operatorname{JSym}}
\newcommand{\op}{\operatorname{Op}}
\newcommand{\tr}{\operatorname{tr}}

\newcommand{\Op}{\operatorname{Op}}

\newcommand{\Ind}{\operatorname{Ind}}
\newcommand{\Tot}{\operatorname{Tot}}
\newcommand{\Mat}{\operatorname{Mat}}
\newcommand{\supp}{\operatorname{supp}}

\newcommand{\Ch}{\operatorname{Ch}}
\newcommand{\ch}{\operatorname{ch}}
\newcommand{\sfz}{\mathsf{z}}
\newcommand{\sfe}{\mathsf{e}}
\newcommand{\Exp}{\operatorname{Exp}}

\newcommand{\calF}{\mathcal{F}}

\newcommand{\calC}{\mathcal{C}}

\newcommand{\doublegroupoid}[4]{\xymatrix{#1  \ar@<2pt>[d] \ar@<-2pt>[d] & #2 \ar@<2pt>[l] 
\ar@<-2pt>[l] \ar@<2pt>[d] \ar@<-2pt>[d] \\ #3  & #4 \ar@<2pt>[l] \ar@<-2pt>[l]}}

\newcommand{\fibprod}[2]{{\,_{#1}\!\!\times\!\!_{#2}\,}}

\title[The transverse index theorem for Lie groupoids]{The transverse index theorem for proper cocompact actions of Lie groupoids}
\author{M.J.~Pflaum}
\address{M.J.~Pflaum: 
Department of Mathematics, University of Colorado,
\newline
Campus Box 395 
Boulder, CO 80309-0395
USA
}     
\email{markus.pflaum@colorado.edu}    
\author{H. Posthuma}    
\address{H. Posthuma: 
         \indent Korteweg-de Vries Institute for Mathematics,
        University of Amsterdam, P.~O.~Box 94248, 1090 GE Amsterdam, 
         The Netherlands
         }
         \email{H.B.Posthuma@uva.nl} 
         
         \author{X.~Tang}
         \address{X.~Tang:  Department of Mathematics, Washington University,
         \newline
         1 brookings Dr.  
           St. ~Louis, MO, 63130, USA}
            \email{xtang@math.wustl.edu}  
\begin{document}
\begin{abstract}
Given a proper, cocompact action of a Lie groupoid, we define a higher index pairing between 
invariant elliptic differential operators and 
smooth groupoid cohomology classes. We prove a 
cohomological index formula for this pairing by applying the van Est map and algebraic 
index theory. Finally we discuss in examples the meaning of the 
index pairing and our index formula.
\end{abstract}
\maketitle
\section*{Introduction}
Index theory of transversely elliptic operators is a subject which has deep roots in several areas 
of mathematics, e.g. the Lefschetz fixed point theorem, the $L^2$-index theorem and discrete 
series representations of Lie groups, the Baum-Connes conjecture and the Novikov conjecture. In 
this article, we will prove a cohomological index theorem for a general 
type of transversely elliptic differential operators associated to a proper, cocompact Lie groupoid 
action. 

The set up we consider is very broad and our main theorem comprises the following special cases:
\begin{itemize}
\item[1)] Atiyah's $L^2$-index theorem for covering spaces in \cite{atiyah}, and the ``higher index'' 
generalization of Connes--Moscovici to higher degree group cocycles in \cite{conmos}.
\item[2)] The $L^2$-index theorem for homogeneous spaces  of Lie groups 
in \cite{cmL2}, and the recent generalization to proper, cocompact actions of Lie groups on 
complete riemannian manifolds in \cite{wang}.
\item[3)] The cohomological index theorem for the pairing between elliptic operators on principal bundles for foliation groupoids and cohomology classes of the associated classifying spaces in \cite{connes,GL,G}.
\item[4)] The longitudinal index theorem on Lie groupoids proved in \cite{ppt} evaluating the natural pairing between differentiable groupoid cohomology and elliptic elements in the universal algebra of the associated Lie algebroid.
\end{itemize} 
Recent developments in noncommutative and differential geometry
  enable us to  improve and generalize these results. 
  
Let us now describe more precisely the setting we consider: Let $\sfG$ be a Lie groupoid over the 
unit space $M$, acting properly on a smooth manifold $Z$ with moment map  $\mu:Z\to M$ given  
by a surjective submersion. Assume that the quotient $Z/\sfG$ is compact. Since $\mu$ is a 
surjective submersion, the kernel of the map $T_\mu:T_zZ\to T_{\mu(z)}M$ for $z\in Z$ defines a 
regular foliation $\calF$ on $Z$ equipped with a natural $\sfG$-action. By a $\sfG$-invariant 
elliptic differential operator on $Z$, we mean a leafwise elliptic differential operator $D$ on $Z$ 
that is invariant under the $\sfG$-action. The analytic index of such a $\sfG$-invariant elliptic 
differential operator on $Z$ was introduced by Connes \cite{connes} (see also Paterson 
\cite{paterson}). Its significance can be seen among other by its connection to the Baum-Connes 
conjecture for Lie groupoids \cite{Tu:bc,hls}. All the recent result listed in 1)--4) above can be 
viewed as computing the pairing between this index class and certain cyclic cocycles in the 
context of noncommutative geometry.

Key to our generalization is to cast all these examples into the framework of groupoids and 
understand the role played by groupoid and Lie algebroid cohomology. This is perhaps least 
obvious in the example 2), but recall that  in the work by Connes and Moscovici \cite{cmL2} the 
$L^2$-index is determined by  integration on the relative Lie algebra cohomology of the Lie 
group $G$ and compact subgroup $H$. Then, one observes that the van Est morphism naturally 
maps the differentiable (or continuous) Lie group cohomology to the relative Lie algebra 
cohomology. So, this suggests that there is a natural pairing between the index of a $G$-invariant 
elliptic differential operator and the differentiable Lie group cohomology. This fits very well with 
our recent work in \cite{ppt}, where a natural morphism from the differentiable groupoid 
cohomology to the cyclic cohomology of the groupoid algebra 
$\calC^\infty_\textup{cpt} (\sfG)$ was constructed. 
Generalizing this idea, we introduce and 
compute natural index numbers of a $\sfG$-invariant elliptic differential operator on $Z$ 
associated differentiable groupoid cohomology classes  of $\sfG$.  

Our main theorem imposes no  restriction on the type of Lie groupoid $\sfG$ or the type of manifold $\sfG$ acts on. 
The initial ingredient of the proof of our index formula is the observation that by cocompactness of the $\sfG$-action $Z$ carries a well behaved calculus $\Psi^\infty_{\rm inv}(Z;\calF)$ of $\sfG$-invariant pseudodifferential operators on $Z$ along the foliation $\calF$. 
Moreover, a $\sfG$-invariant elliptic differential operator $D$ on $Z$ has a well defined index $\Ind(D)$ which is an element in the $K$-theory of the algebra $\Psi^{-\infty}_{\text{inv}}(\sfG; Z)$ of $\sfG$-invariant smoothing operators on $Z$. Next, we construct a canonical pairing between the differentiable groupoid cohomology $H^\bullet_{\text{diff}}(G;L)$ with values in a $\sfG$-representations $L$ of ``transversal densities'', and the $K$-theory of $\Psi^{-\infty}_{\text{inv}}(\sfG;Z)$,
\begin{equation}\label{eq:index-pairing}
\left<\ ,\ \right>: H^\bullet_{\text{diff}}(\sfG;L)\times K_0(\Psi^{-\infty}_{\text{inv}}(\sfG;Z))\to \C.
\end{equation}
The pairing of $\Ind(D)$ with classes in $H^\bullet_{\text{diff}}(\sfG)$ defines a family of index numbers of $D$. When $\sfG$ is a unimodular Lie group $G$,  a bi-invariant density $\Omega$ defines a class in $H^0(G;L)$ and gives rise to a trace on the $G$-invariant smoothing operators, and therefore the pairing $\left< [\Omega], \Ind(D)\right>$ of $\Ind(D)$  recovers the $L^2$-index of a $G$-invariant elliptic differential operator.    

We compute the cohomological formulas of the above index numbers using the localized groupoid index theory developed in \cite{ppt}. By localizing the support of the kernel of a smoothing operator to the diagonal, we introduce the localized K-theory  group $K_0^{\rm loc} (\Psi^{-\infty}_{\rm inv}(Z;\calF) )$.  And given a $\sfG$-invariant elliptic differential operator $D$ on $Z$, we introduce a localized index $\Ind_{\text{loc}}(D)$ as an element in 
$K^{\rm loc}_0 (\Psi^{-\infty}_{\rm inv}(Z;\calF) )$ which is mapped to $\Ind(D)\in K_0(\Psi^{-\infty}_{\rm inv}(Z;\calF) )$ under the natural forgetful map from $K_\bullet^{\rm loc} (\Psi^{-\infty}_{\rm inv}(Z;\calF) )$ to $K_\bullet (\Psi^{-\infty}_{\rm inv}(Z;\calF) )$.  Through the van Est morphism $\Phi_Z$, the localized version of the differentiable groupoid cohomology for the proper $\sfG$-action on $Z$ is the $\sfG$-invariant foliated cohomology $H^\bullet_{\calF}(Z;\mu^*L)^\sfG$ defined by the de Rham cohomology on the $\sfG$-invariant leafwise differential forms. We will define a canonical pairing 
\[
\left<\ ,\ \right>_{\rm loc}: H^\bullet_{\calF}(Z;\mu^*L)^\sfG \times K_0^{\rm loc}(\Psi^{-\infty}_{\text{inv}}(Z;\calF))\to \C,
\]
which is compatible with the pairing  in Eq. (\ref{eq:index-pairing}), i.e.,
\begin{equation}\label{eq:local-pairing}
\left< \Ind(D), \alpha\right>=\left<\Ind_{\rm loc}(D), \Phi_Z(\alpha)\right>_{\rm loc},
\end{equation}
where $\Phi_Z: H^\bullet_{\rm diff}(\sfG)\to H^\bullet_\calF(Z;\mu^*L)^\sfG$ is the van Est map, and $\alpha$ is a class in $H^\bullet_{\rm diff}(\sfG)$. 

The main theorem of this paper is the following cohomology formula of the index number 
introduced above:
 \begin{equation}\label{eq:index-formula}
 \left<\Ind_{\rm loc}(D), \alpha\right>_{\rm loc}:=\frac{1}{(2\pi\sqrt{-1})^k}\int_{T^*_\mu Z}(c\circ\pi)\pi^*\alpha\wedge \hat{A}(\calF^*)\wedge\ch(\sigma(D)),
 \end{equation}
for $\alpha \in H^{\rm even}_\calF(Z; \mu^*L)^\sfG$. Here, $\pi$ is the natural projection from $T^*_\mu Z$ to $Z$, and $c(x)$ is a cut-off function for the proper cocompact $\sfG$ action on $Z$, and $\hat{A}(\calF^*)$ is the $\hat{A}$ class of $T^*_\mu \calF$ as a foliated bundle over $\calF$, and 
finally $\ch(\sigma(D))$ is the leafwise Chern character of the $K$-theory class defined by the principal symbol $\sigma(D)$.  The integral is actually independent of the choice of cut-off function $c$.

This paper is organized as follows. In Section \ref{preli}, we will introduce the van Est map for a proper cocompact $\sfG$ action on $Z$, and also introduce the calculus of $\sfG$-invariant pseudodifferential operators. The index pairing between (localized) differentiable groupoid cohomology of $\sfG$ and the (localized) $K$-theory of $\sfG$-invariant smoothing operators is explained in Section \ref{sec:pairing}. In the same section, we introduce the index $\Ind(D)$ and the 
localized index $\Ind_{\rm loc}(D)$ of a $\sfG$-invariant elliptic differential operator $D$. In Section \ref{sec:index}, we prove the index formula Eq. (\ref{eq:index-formula}) by the algebraic index theory of deformation quantization. Interesting examples of our index theorem are discussed in Section \ref{sec:example}.  In particular, an example of a foliation index theorem on orbifolds is explained. 

\medskip

\noindent{\bf Acknowledgments:} We would like to thank Nigel Higson, Hang Wang, and Weiping Zhang for inspiring discussions on index formulas for elliptic $\sfG$-invariant pseudodifferential operators. 
Pflaum is partially supported by NSF grant DMS 1105670, and 
Tang is partially supported by NSF grant DMS 0900985. 

 \section{Preliminaries}
 \label{preli}
  \subsection{Proper actions of Lie groupoids}
\label{sec_properaction}
 Let $\sfG\rightrightarrows M$ be a Lie groupoid. In this paper we have the convention that we draw 
 arrows in $\sfG$ from left to right: $x\stackrel{g}{\rightarrow} y$ for an arrow $g\in\sfG$ with source 
 $s(g)=x$ and target $t(g)=y$. This means that the multiplication $g_1g_2$ of two arrows $g_1, 
 g_2\in\sfG$ is defined when $t(g_1)=s(g_2)$. The unit map of the groupoid will be denoted by $u: M \rightarrow \sfG$. For fixed $x\in M$, we write $\sfG^x$ for the submanifold of all arrows $g\in \sfG$ with $s(g)=x$, and $\sfG_x$ for all arrows with $t(g)=x$.
 Associated to the Lie groupoid $\sfG$ is its Lie algebroid which in the following will be denoted by
$\sfA$.
 
 A left action of $\sfG$ on a manifold $Z$ is given by a submersion $\mu:Z\to M$, called the {\em moment map}, which controls the action of $\sfG$ given by a smooth map
 \[
    \sfG\fibprod{t}{\mu} Z\to Z,
 \]
 where $ \sfG\fibprod{t}{\mu} Z:=\{(g,z),~t(g)=\mu(z)$. (This is our general notation for fibered product 
 of manifolds.) We write the action above simply as $(g,z)\mapsto g z$. This map should satisfy the usual axioms for an action, i.e., $g_1(g_2z)=(g_1g_2)z$ and $u(x)z = z$, whenever defined,  and $\mu(gz)=s(g)$.

We denote by $\calF$ the foliation on $Z$ by the fibers of $\mu$. The foliation is given by the (involutive) sub-bundle
$T_\mu Z \subset TZ$ consisting of all  $X \in TZ$ which are tangent to the fibers of $\mu$ or in other words which are
tangent to the leaves of the foliation $\calF$. Note that, here, by slight abuse of language, we identify the leaves of
the foliation with the fibers of the moment map $\mu$, even if the fibers are not connected.  

We write $T^*_\mu Z$ for the bundle dual to
$T_\mu Z$. The canonical projection of $T_\mu Z$ (resp.~$T_\mu^* Z$) onto $Z$ is denoted 
by $\pi_{T_\mu Z}$ (resp.~$\pi_{T_\mu^* Z}$). Since $g\in \sfG$ induces a diffeomorphism from $\mu^{-1} \big( t(g) \big)$ to
$\mu^{-1} \big( s (g)\big)$,  in the following also will be denoted by $g$,  there is a natural 
action of $\sfG$ on $T_\mu Z$ and $T^*_\mu Z$ by the differentiation.

\begin{definition}
\label{cut-off}
A \emph{cut-off density on $Z$ adapted to the $\sfG$-action} is
a nowhere vanishing smooth section $c\in\Gamma_\textup{cpt}^\infty(Z, \left| \bigwedge^\textup{top}\right| \sfA^*)$ satisfying
\begin{enumerate}[(i)]
\item
  For every compact $K\subset M$, the set $\supp (c) \cap \sfG(K)$ is compact.
\item  
 \label{eq-cut-off}
 $\int\limits_{\sfG^{\mu(z)}}c(g^{-1}z)=1$, for all $z\in Z$.
\end{enumerate}
\end{definition}
If the $\sfG$-action on $Z$ is cocompact, the first condition just means that the support of $c$ is 
compact. Notice that a cut-off density on $Z$ can be regarded as a (left)  Haar system on the 
transformation groupoid $\sfG \ltimes Z$.
Therefore, it follows by  \cite[Sec.~1.7]{Tu:bc} that such a density exists if and only if the action of 
$\sfG$ on $Z$ is proper.

Throughout this paper, we fix an \emph{invariant leafwise riemannian metric} on $Z$, i.e. a riemannian metric $\eta$ 
on $T_\mu Z$ such that $\eta (gX,gY) = \eta (X,Y)$ for all $X,Y \in T_\mu Z$ and $g\in G$ with
$\mu\big( \pi_{T_\mu Z} (X)\big)=\mu \big( \pi_{T_\mu Z} (Y)\big)= t(g)$.
If a cut-off density $c$ on $Z$ is given, and $\varrho$ an arbitrary  riemannian metric on $Z$,   we obtain an
invariant leafwise riemannian metric on $Z$ by putting
\[
    \eta (X,Y) := \int_{\sfG^{\mu (z)}} \varrho (g^{-1}X,g^{-1}Y) c(g^{-1}z)  \quad \text{for $z\in Z$, and 
    $X,Y\in T_{\mu,z}  Z$}\: .
\] 
Since by the observation above $c$ is a left Haar system on $\sfG\ltimes Z$, one easily checks that this metric 
is invariant.
\subsection{The van Est morphism} 
\label{vanest}
For any proper action of a Lie groupoid $\sfG$ on $Z$ with moment map $\mu:Z\to M$, 
Crainic \cite{crainic} has constructed, for any representation $E$ of $\sfG$, a ``van Est'' map
\[
\Phi_Z:H^\bullet_{\rm diff}(\sfG;E)\to H^\bullet_\calF(Z;\mu^*E)^\sfG,
\]
from the differentiable groupoid cohomology $H^\bullet_{\rm diff}(\sfG;E)$ with values in $E$ to 
the $\sfG$-invariant part of the foliated de Rham cohomology $ H^\bullet_\calF(Z;\mu^*E)$ with values in the flat vector 
bundle $\mu^*E$. In this section, we give an alternative construction of this map using a so-called 
double groupoid defined by the action to construct a morphism on the level of cochain 
complexes. In fact, a generalization of this morphism exists for any action, not necessarily proper ones.

The action of $\sfG$ on $Z$ gives a natural example of a double groupoid, which is essentially a 
groupoid in the category of groupoids, c.f. \cite{bm,MT}. More concretely, we have a square 
of groupoids
\begin{equation}\label{square}
\doublegroupoid{\sfG{_t\times_\mu} Z}{\sfG{_t\times_\mu} (Z{_\mu \times_\mu}Z)}{Z}{Z{_\mu \times_\mu}Z}.
\end{equation}
The manifold $Z{_\mu \times_\mu}Z$ has a natural map to $M$ by mapping $(z_1, z_2)$ to $\mu(z_1)=\mu(z_2)$ in $M$, and we denote this map simply also by $\mu$.  This explains the definition of $\sfG{_t\times_\mu} (Z{_\mu \times_\mu}Z)$. As both $\mu$ and $t$ are assumed to be surjective submersions, the fiber products in Square \ref{square} are all smooth manifolds. Furthermore, observe that $Z{_\mu \times_\mu}Z$ has a natural groupoid structure with the unit space $Z$. Similarly, there is a natural map $\nu$ from $\sfG{_{t}\times _\mu}Z$ to $\sfG$, and $\sfG{_t\times_\mu} (Z{_\mu \times_\mu}Z)$ can be identified as the fiber product of $\sfG{_{t}\times _\mu}Z$ with itself over the map $\nu$ and therefore is equipped with a natural groupoid structure. This way, the two horizontal edges of the square (\ref{square}) are  Lie groupoids. The $\sfG$ action on $Z$ lifts to the diagonal action of $\sfG$ on $Z{_\mu \times_\mu}Z$. The $\sfG$-actions on $Z$ and $Z{_\mu \times_\mu}Z$ equip ${\sfG{_t\times_\mu} Z}$ and $\sfG{_t\times_\mu} (Z{_\mu \times_\mu}Z)$ with action groupoid structures, and make the two vertical edges of the square (\ref{square}) into groupoids. It is straightforward to check the above square of groupoids does satisfy the definition of a full double Lie groupoid, c.f. \cite[Defn. 3.1, 3.2]{MT}. 

As is explained in \cite[Prop. 3.10]{MT}, the nerve functor on a full double Lie groupoid produces a bisimplicial manifold $N^{\bullet, \bullet}$. In our example (\ref{square}), this bisimplicial manifold is defined as follows,
\[
N^{p,q}(\sfG;Z):=\sfG^{(p)}\fibprod{\tau_p}{\mu_q} Z^{(q)}_\mu,
\]
where
\[
\sfG^{(p)}:=\underbrace{\sfG \fibprod{t}{s}\ldots \fibprod{t}{s}\sfG}_{p+1~ \rm copies},\quad
Z_{\mu}^{(q)}:=\underbrace{Z\fibprod{\mu}{\mu}\ldots \fibprod{\mu}{\mu} Z}_{q+1 ~\rm copies}.
\]
The maps $\tau_p:\sfG^{(p)}\to M$ and $\mu_q:Z_{\mu}^{(q)}\to M$ are defined by
\[
\tau_p(g_0,g_1, \cdots, g_p):=t(g_p),\qquad \mu_q(z_0, \cdots, z_q)=\mu(z_0).
\]
The simplicial structure on $\sfG^{(\bullet)}\fibprod{\tau}{\mu_q} Z_{\mu}^{(q)}$ is lifted from the simplicial structure on the nerve space $\sfG^{(\bullet)}$ of the groupoid $\sfG$, and the simplicial structure on $\sfG^{(p)}\fibprod{\tau_p}{\mu} Z_\mu^{(\bullet)}$ is lifted from the simplicial structure on the nerve space $Z_\mu^{(\bullet)}$ of the groupoid $Z\fibprod{\mu}{\mu}Z\rightrightarrows Z$. Accordingly, the space $C^{p,q}_{\rm diff}(\sfG;Z):=\calC^\infty(N^{p,q}(\sfG;Z))$ of smooth functions on $N^{p,q}(\sfG;Z)$ forms a bicosimplicial complex. Define the differentiable cohomology of the action of $\sfG$ on $Z$ to be the cohomology of the total complex of $C^{\bullet,\bullet}_{\rm diff}(\sfG;Z)$, which is quasi-isomorphic to the diagonal by the Eilenberg--Zilber theorem. This cohomology is denoted by $H^\bullet_{\rm diff}(\sfG;Z)$.

The differentiable Lie groupoid cohomology $H^\bullet_{\rm diff}(\sfG)$ of $\sfG$ is defined as the cohomology of the simplicial complex $C^\bullet_{\rm diff}(\sfG):=\calC^\infty(\sfG^{(\bullet)})$. There is a natural forgetful map from $N^{\bullet,\bullet}(\sfG; Z)$ to the nerve space $\sfG^{(\bullet)}$ of $\sfG$. This map induces a natural morphism $\alpha$ from $C^\bullet_{\rm diff}(\sfG)$ to $C^{\bullet,\bullet}_{\rm diff}(\sfG;Z)$, and therefore a morphism 
\begin{equation}\label{eq:alpha}
\alpha:H^\bullet_{\rm diff}(\sfG)\rightarrow H^\bullet_{\rm diff}(\sfG;Z). 
\end{equation}

\begin{proposition}\label{prop:alpha} The map $\alpha$ defined by Eq. (\ref{eq:alpha}) is an isomorphism.
\end{proposition}
\begin{proof}
Consider the spectral sequence associated to the $p$-filtration on $\calC^\infty(N^{\bullet, \bullet}(\sfG;Z))$. The $E_2$-term of the spectral sequence can be computed
\[
E_2^{p,0}=H^{p}_{\rm diff}(\sfG),\qquad E_1^{p,q}=0,\ q\geq 1. 
\]
This follows from the fact that the complex $C^{p,\bullet}_{\rm diff}(\sfG;Z)$ identifies with the 
complex computing groupoid cohomology of the Lie groupoid $Z\fibprod{\mu}{\mu}Z\rightrightarrows Z$ with values in $C^p_{\rm diff}(\sfG)$. Since $Z\fibprod{\mu}{\mu}Z\rightrightarrows Z$ is a proper groupoid, its cohomology is concentrated in degree zero by \cite[\S 2.1.]{crainic} and one finds the result for $E_2^{p,q}$.
Therefore the spectral sequence degenerates at $E_2$. We notice that $E_1^{p,0}$ is exactly the complex of differentiable groupoid cohomology of $\sfG$, and therefore $\alpha$ is an isomorphism. 
\end{proof}

We now give an alternative complex computing $H^\bullet_{\rm diff}(\sfG;Z)$. The action of the 
groupoid $\sfG$ on  the groupoid $Z{_\mu\times_\mu}Z\rightrightarrows Z$ defines a natural action 
of $\sfG$ on the $Z_\mu^{(\bullet)}$, and we define 
$C^p_{\rm diff}(Z)^\sfG:=\calC^\infty_{\rm inv}(Z_\mu^{(p)})$. This defines a graded vector space $C^\bullet_{\rm diff}(Z)^\sfG$ which inherits a differential $d$ defined by
\begin{equation}
\label{diff-def}
d\varphi(z_0,\ldots,z_{k+1}):=\sum_{i=0}^{k+1}(-1)^i\varphi(z_0,\ldots,\hat{z}_i,\ldots,z_{k+1}),
\end{equation}
where $\varphi\in C_{\rm diff}^k(Z)^\sfG$ and $\hat{z}_i$ means we omit $z_i$ from the argument 
of $\varphi$. The corresponding cohomology is denoted by $H_{\rm diff}^\bullet(Z_\mu)^\sfG$. 

\begin{corollary}\label{cor:gpdcoh} When the $\sfG$-action on $Z$ is proper, 
$H^\bullet_{\rm diff}(\sfG;Z)$ is isomorphic to $H_{\rm diff}^\bullet(Z_\mu)^\sfG$.
\end{corollary}
\begin{proof}
We consider the spectral sequence associated to the filtration of the bicomplex $\calC^\infty(N^{\bullet, \bullet}(\sfG;Z))=\bigoplus_{p,q}C^\infty(N^{p,q}(\sfG;Z))$ with respect to the degree $q$: $E_2^{r,s}(\sfG;Z)$ is computed to be
\[
E_2^{0,s}=H^s(\calC^\infty_{\rm inv}(Z_\mu^{(\bullet)}), d),\qquad E_2^{r,s}=0,\ s\geq 1,
\]
this time because the groupoid $\sfG\ltimes Z\rightrightarrows Z$ is proper.
Therefore, the spectral degenerates at $E_2$ again. It follows that the cohomology $H^\bullet_{\rm diff}(\sfG;Z)$ is isomorphic to $H^\bullet_{\sfG}(Z_\mu)$.
\end{proof}

Crucial for the construction of the van Est map is a localized version of the bicomplex $\calC^\infty(N^{(\bullet, \bullet)}(\sfG; Z))$. 
Notice that $\sfG^{(p)}\fibprod{\tau_p}{\mu}Z$ is embedded inside $\sfG^{(p)}{_{\tau_p} 
\times_{\mu_q}} Z_{\mu}^{(q)}$ via the canonical diagonal inclusion of $Z\hookrightarrow Z^{(q)}
$. This embedding is also compatible with the simplicial structures. Define
\[
C^{p,q}_{\rm loc}(G;Z):=\operatorname{germs}_{\sfG^{(p)}{_{\tau_p}\times_{\mu_q}}Z}\left(\calC^\infty \left(\sfG^{(p)}{_{\tau_p}\times_{\mu_q}}Z^{(q)}_\mu\right)\right),
\]
taking the germ of a smooth function on $\sfG^{(p)}{_{\tau_p}\times_{\mu_q}}Z^{(q)}_\mu$ at the 
embedded submanifold $\sfG^{(p)}\fibprod{\tau_p}{\mu}Z$.
There is a natural morphism 
\[
L:C^{\bullet,\bullet}_{\rm diff}(\sfG;Z))\to C^{\bullet,\bullet}_{\rm loc}(\sfG;Z),
\]
by taking germs at $\sfG^{(p)}\fibprod{\tau_p}{\mu}Z$, 
and we equip the right hand side with the induced differentials turning 
it into a bicosimplicial complex. Denote the inherited differentials by $d_\sfG$ and $d_Z$. The 
cohomology of the total complex of this bicomplex is denoted by $H^\bullet_{\rm loc}(\sfG;Z)$. 
The morphism $L$ above induces a map 
\begin{equation}\label{eq:L}
L:H^\bullet_{\rm diff}(\sfG;Z)\rightarrow H^\bullet_{\rm loc}(\sfG;Z).
\end{equation}

On $Z$,  let $\Omega^\bullet_\calF$ be the space of leafwise differential forms on $Z$ with 
respect to the foliation $\calF$. This forms a fine sheaf over $Z$ with a natural $\sfG$ action 
induced by the $\sfG$ action on $Z$. To compute $H^\bullet_{\rm loc}(\sfG;Z)$, we will consider 
the following bicomplex,
\[
C^{p, q}_{\rm diff}(\sfG;\calF):=\Gamma\left(\sfG^{(p)}\fibprod{\tau_p}{\mu} Z, \beta^*\Omega^q_\calF\right),
\]
where $\beta:\sfG^{(p)}{_{\tau_p}\times_\mu Z}\to Z$ is defined by $\beta(g_0, \cdots, g_p;z)=z$.
 The cohomology of the total complex of $C^{p,q}_{\rm diff}(\sfG;\calF)$ is denoted by
  $H^\bullet_{\rm diff}(\sfG; \calF)$.
\begin{proposition}\label{prop:vanest}
There is a natural isomorphism
\[
\lambda: H^\bullet_{\rm loc}(\sfG;Z)\stackrel{\cong}{\longrightarrow} H^\bullet_{\rm diff}(\sfG;\calF).
\]
\end{proposition}
\begin{proof}
This is just a $\sfG$-equivariant version of the well-known isomorphism between Alexander--Spanier cohomology and de Rham cohomology, see e.g. \cite[\S 1]{conmos}. 

We point out that on $Z$, there is a fine (pre)sheaf of complexes $\mathsf{C}_{\rm AS}^\bullet$ defined by 
\[
\mathsf{C}_{\rm AS}^{k}:=\operatorname{germs}_{Z}\left(\calC^\infty(Z^{(k)}_\mu)\right),
\]
where $Z$ is embedded diagonally in $Z^{(k)}_\mu$.  And the differential $d_Z$ on $\calC^\infty(Z^{(k)}_\mu)$ naturally descends to the differential $d_Z$ on $\mathsf{C}_{\rm AS}^\bullet$. The space $C^{p,q}_{\rm loc}(\sfG;Z)$ is identified with the space of global sections of the sheaf $\beta_p^*\mathsf{C}_{\rm AS}^q$ on $\sfG^{(p)}\fibprod{\tau_p}{\mu}Z$. From the $\sfG$ action on $Z$, the sheaf $\mathsf{C}_{\rm AS}^\bullet$ is equipped with a canonical $\sfG$ action. As the sheaf of complexes $\beta_q^*\mathsf{C}_{\rm AS}^\bullet$ is fine, the cohomology of $C^{\bullet,\bullet}_{\rm loc}(\sfG;Z)$ is isomorphic to the groupoid differentiable cohomology of the groupoid $\sfG\fibprod{\tau}{\mu}Z\rightrightarrows Z$ with the coefficient $\mathsf{C}_{\rm AS}^\bullet$. 

For $C^{\bullet,\bullet}_{\rm diff}(\sfG; \calF)$, we consider  a sheaf $\Omega^\bullet_\calF$ on $Z$ of leafwise differential forms along $\calF$ with the de Rham differential $d_\calF$.  And $C^{p,q}_{\rm diff}(\sfG; \calF)$ is the space of global sections of the sheaf $\beta_p^*\Omega^q_\calF$. This sheaf is also equipped with a canonical $\sfG$ action. The cohomology of $C^{\bullet,\bullet}_{\rm diff}(\sfG; \calF)$ is isomorphic to the groupoid differentiable cohomology of $\sfG\fibprod{\tau}{\mu}Z\rightrightarrows Z$ with the coefficient $\Omega^
\bullet_\calF$.

There is a natural quasi-isomorphism  $\lambda:\mathsf{C}_{\rm AS}^\bullet\to \Omega^{\bullet}_\calF$ of sheaves of differential complexes on $Z$ as is explained in \cite[Lem. 1.5]{conmos}  by
\[
\lambda(f_0\otimes \cdots\otimes f_k)=f_0df_1\wedge\cdots \wedge df_k. 
\]
It is natural to check that $\lambda$ is $\sfG$-equivariant and therefore defines a quasi-isomorphism
\[
\lambda: \left(\Gamma(\sfG^{(p)}\fibprod{\tau_p}{\mu}Z, \beta_p^*(\mathsf{C}_{\rm AS}^q)), d_Z\right)\longrightarrow \left(\Gamma(\sfG^{(p)}\fibprod{\tau_p}{\mu}Z, \beta^*_p(\Omega^q_\calF)),d_{\calF}\right),
\]
inducing the desired isomorphism on cohomology.
\end{proof}

\begin{lemma}\label{cor:vanest-proper}When $\sfG$ acts on $Z$ properly, $H^\bullet_{\rm diff}(\sfG; \calF)$ is equal to $H^\bullet_\calF(Z)^\sfG$, the cohomology of the de Rham differential on the space of $\sfG$-invariant leafwise differential forms on $Z$ with respect to $\calF$. 
\end{lemma}
\begin{proof}
We apply the spectral sequence associated to the filtration of degree $q$ on the bicomplex $C^{\bullet,\bullet}_{\rm diff}(\sfG; \calF)$.  At the $E_1$ level, $E^{\bullet, q}_{1}$ is exactly the complex of differentiable cohomology of the groupoid $\sfG{_t\times_\mu}Z$ with coefficient in $\Omega^q_{\calF}(Z)$, the leafwise differential forms on $Z$ with respect to $\calF$. Since the  $\sfG$-action on $Z$ is proper, $\sfG{_t\times_\mu}Z$ is a proper groupoid. By \cite[Prop. 1]{crainic}, $E^{\bullet, q}_{1}$ is computed to be
\[
E^{0, q}_{1}=\Omega^q_\calF(Z)^\sfG, \qquad E^{p,q}_{1}=0, p\geq 1.
\]
Therefore, the spectral sequence degenerates at the $E_2$-level, and therefore  $H^\bullet_{\rm diff}(\sfG; \calF)$ and $H^\bullet_\calF(Z)^\sfG$ are isomorphic.
\end{proof}
 Composing the isomorphism $\alpha$ defined in Eq. (\ref{eq:alpha}) with the morphism 
$L$ defined in Eq. (\ref{eq:L}) and the isomorphism $\lambda$ in Prop. \ref{prop:vanest}, we  
define the following morphism of complexes:
\[
\Phi_Z:=\lambda\circ L\circ\alpha: C^\bullet_{\rm diff}(\sfG)\to \Tot^\bullet\left(C^{\bullet,\bullet}_{\rm diff}(\sfG;\calF)\right).
\]
We now have the following:
 \begin{theorem}\label{thm:vanest}
 Let $\sfG$ act on a manifold $Z$ such that the moment map $\mu:Z\to M$ is a surjective 
 submersion. When the action is proper, the morphism $\Phi_Z$, combined with the isomorphism of 
 Lemma \ref{cor:vanest-proper} induces a map
 \[
 \Phi_Z: H^\bullet_{\rm diff}(\sfG)\to H^\bullet_{\calF}(Z)^\sfG,
 \]
 which coincides with the van Est map of \cite{crainic}.
 \end{theorem} 
 \begin{remark}\label{rmk:vainest}
 Theorem 3 of \cite{crainic} states that the van Est map $\Phi_Z$ is an isomorphism in degree 
 $\bullet\leq n$ and injective for $\bullet=n+1$, when the $\sfG$-action is proper and the fibers of 
 the moment map $\mu:Z\to M$ are homologically $n$-connected. Theorem \ref{thm:vanest} 
 above provides a natural framework to  generalize this van Est Theorem to the case of a general 
 Lie groupoid action. Indeed, by studying when the map 
 $L:H^\bullet_{\rm diff}(\sfG;Z)\to H^\bullet_{\rm loc}(\sfG;Z)$ is an isomorphism, we can obtain 
 similar conditions on the moment map $\mu:Z\to M$ for the van Est morphism $\Phi_Z$ to be an 
 isomorphism. 
 \end{remark}

\subsection{The invariant pseudodifferential calculus}
\label{sec_pseudo}

 Here we provide the construction of the algebra of invariant pseudodifferential 
 operators on a manifold $Z$ along the fibers of $\mu$, where $\sfG\rightrightarrows M$ is a Lie groupoid acting 
properly on $Z$ with moment map $\mu:Z  \to M$.
This pseudodifferential calculus 
 extends the one on Lie groupoids in \cite{nwx}, to which it reduces when $Z=G$, or, more 
 generally when the action is free: in this case, we have a principal $\sfG$-bundle, and we simply  
 pass to the Morita equivalent gauge groupoid. For actions of Lie groups, the invariant 
 pseudodifferential calculus was first constructed in \cite[\S 1]{cmL2}.
 
With notation from Section \ref{sec_properaction}, consider the space $\sym^m_\textup{inv}  (Z; \calF)  $ of \emph{invariant symbols over $Z$ along $\calF$ of order $m$}. By definition, this is the space of all $a\in \calC^\infty (T^*_\mu Z)$ such that the following conditions hold true:
\begin{enumerate}[(i)]
\item
  The function $a$ is $\sfG$-invariant in the sense that $ a (g \xi ) = a (\xi)$ for all $\xi \in T^*Z$ and $g \in G$ with
  $\mu(\pi(\xi)) = t(g)$. 
\item
  With respect to some local coordinates $x: U \rightarrow \R^n$ over some coordinate patch $U\subset Z$ 
  and the induced local coordinate system $(\sfz,\zeta ) : T^*U \rightarrow \R^{2n}$ there exists for all compact     
  $K\subset U$ and $\alpha , \beta \in \N^n$ a constant $C_{K,\alpha,\beta} > 0$ such that for all $ \xi \in T^*U$ with
  $\pi_{T^* Z} (\xi) \in K$ the estimate  
  \[
       \left| \partial_\sfz^\alpha \partial_{\zeta}^\beta a(\xi) \right| \leq C_{K,\alpha,\beta} 
      ( 1 + \| \zeta ( \xi )  \|^2)^{m/2}  
  \]
  holds true. 
\end{enumerate}
As usual, we put 
\[ 
   \sym^\infty_\textup{inv}  (Z; \calF) = \bigcup\limits_{m\in \Z} \sym^m_\textup{inv}  (Z; \calF) 
   \]
   and also 
   \[
   \sym^{-\infty}_\textup{inv}  (Z; \calF)  = \bigcap\limits_{m\in \Z} \sym^m_\textup{inv}  (Z; \calF)  \: .
\]
To define the calculus of (invariant) pseudodifferential operators along $\calF$ recall that we have fixed an invariant leafwise riemannian  metric $\eta$ on $T_\mu Z$. This enables us to construct a linear map 
$\Op : \sym^\infty _\textup{inv}  (Z; \calF) \rightarrow  
\Hom \big( \calC^\infty_\textup{$\sfG$-cpt} (Z) ,\calC^\infty_\textup{$\sfG$-cpt}  (Z)\big) $,
called the \emph{quantization map}, as follows. To this end note first that $\eta$ induces a connection 
on $T_\mu Z$ and an exponential function $\Exp : W \rightarrow Z$, where $W$ is an appropriate open neighborhood
of the zero section of $T_\mu Z$, and $\Exp X$ is the end point of the unique geodesic 
$\gamma: [0,1] \rightarrow L_z \subset Z$  within the leaf $L_z$ through the point $z = \pi_{T_\mu Z} (X)$
such that $\gamma (0) = z$ and $\dot\gamma(0) = X$. 
Now choose a smooth cut-off function $\chi :Z \times Z \rightarrow [0,1]$ which equals $1$ on a neighborhood of 
the diagonal, has support in a necessarily (larger) neighborhood of the diagonal over which $(\pi_{T_\mu Z},\Exp)^{-1}$ 
is defined, and satisfies $\chi(y,z) =\chi(z,y)$. Then let $\sfe: Z\times T^*_\mu Z \rightarrow \C$ be  the map
\[
   (y,\xi) \mapsto \chi (y,z(\xi)) e^{\sqrt{-1} \langle \xi , \Exp_{z(\xi)}^{-1} (y) \rangle} \quad 
   \text{where $z (\xi) = \pi_ {T^*_\mu Z} (\xi)$} \: .
\]
Given $a \in \sym^m_\textup{inv} (Z;\calF)$, $f\in \calC^\infty_\textup{cpt} (Z)$ and $z\in Z$, we now put
\begin{equation}
\label{quantization}
   \Op ( a) f (z) =  \int_{T_{\mu,z} Z} \int_{L_z} \sfe (y,z ) a(\xi) f(y)  \,  dy   \,  d\xi \: ,
\end{equation}
where, up to a factor $(2\pi \sqrt{-1})^{-r/2}$ with $r$ being the dimension of the leaves, 
integration with respect to $y \in L_z$ is over the natural volume density $d\nu$ induced by the riemannian metric $\eta_{|L_z}$, and integration over $  T_{\mu,z} Z$ is with respect to the natural translation invariant  measure given by
the positive symmetric bilinear form $\eta_z$.

Next, consider the fibered product $Z \fibprod{\mu}{\mu}  Z$ and note that $\sfG$ acts on it by the diagonal action.  
Given a smooth $\sfG$-invariant $\sfG$-compactly supported  function $k$ on $Z\fibprod{\mu}{\mu} Z$ we obtain an operator
\[
      P_k :  \calC^\infty_\textup{$\sfG$-cpt} (Z) \rightarrow  \calC^\infty_\textup{$\sfG$-cpt} (Z),
      \quad f \mapsto \Big( Z \ni z \mapsto \int_{L_z}  k(z,y)   f(y) dy \in \C \Big),
\]
where integration over $L_z$ is with respect to  the same volume density $d\nu$  as above. The space of operators
$P_k \in \Hom \big( \calC^\infty_\textup{$\sfG$-cpt} (Z) ,\calC^\infty_\textup{$\sfG$-cpt}  (Z)\big) $ obtained in this way
are called the \emph{invariant smoothing operators on $Z$ along $\calF$}.
 
By an \emph{invariant pseudodifferential operator over $Z$ along $\calF$ of order $m$} we now understand 
an operator $A\in \Hom \big( \calC^\infty_\textup{$\sfG$-cpt} (Z) ,\calC^\infty_\textup{$\sfG$-cpt}  (Z)\big) $, which up to an invariant smoothing operator coincides with $\Op (a)$ for some $a \in \sym^m_\textup{inv} (Z;\calF)$. 
Denote by $\Psi^m_\textup{inv}(Z;\calF)$ the space of such pseudodifferential operators, and put as usual
\[
  \Psi^\infty_\textup{inv} (Z;\calF):=\bigcup_{m\in\Z}\Psi^m_\textup{inv}(Z;\calF),\quad \text{and} \quad 
  \Psi^{-\infty}_\textup{inv} (Z;\calF):=\bigcap_{m\in\Z}\Psi^m_\textup{inv}(Z;\calF)\: .
 \]
Elements of $\Psi^m_\textup{inv}(Z;\calF)$ can now be regarded as families $\{P_x\}_{x\in M}$, where each $P_x$ is a 
pseudodifferential operator on $\mu^{-1}(x)$ of order $m$, satisfying:
 \begin{enumerate}[(i)]
 \item the family $\{P_x\}_{x\in M}$ is smooth in its dependence on $x\in M$,
 \item the family $\{P_x\}_{x\in M}$ is $\sfG$-invariant:
 \[
  P_{s(g)}=L_g\circ P_{t(g)}\circ L_{g}^{-1},
 \]
 where $L_g:C^\infty(\mu^{-1}(s(g)))\to C^\infty(\mu^{-1}(t(g)))$ is the pull-back along  
 diffeomorphism induced by the action of $g\in\sfG$,
 \item the support of $P$, defined as
 \[
   \operatorname{supp}(P):=\overline{\bigcup_{x\in M}\operatorname{supp}(P_x)}\subset Z \fibprod{\mu}{\mu} Z,
 \]
    is $\sfG$-compact.
 \end{enumerate}
 This implies that the invariant pseudodifferential operator calculus we defined coincides with the one 
 by \cite[\S 4]{paterson}.
 Moreover, it follows by the exposition above (cf.~also \cite[\S 4]{paterson}) that 
 $ \Psi^\infty_\textup{inv} (Z;\calF)$ forms a 
 filtered algebra containing the smoothing operators in $\Psi^{-\infty}_\textup{inv} (Z;\calF)$ as an  ideal.
 
 The quantization map $\Op$ has a quasi-inverse, the so-called \emph{symbol map} 
$\sigma : \Psi^\infty_\textup{inv} (Z;\calF) \rightarrow \sym^\infty_\textup{inv} (Z;\calF)$ which is given by 
\[
   \sigma (P) (\xi) := P\big( \sfe ( - , z(\xi) \big)  (z(\xi))  \quad \text{for $P\in \Psi^\infty_\textup{inv} (Z;\calF)$ and
   $\xi \in T^*_\mu Z$}. 
\]

 \section{The analytic index pairing}\label{sec:pairing}
 \subsection{The trace}
 In this section we shall construct a trace on the algebra $\Psi^{-\infty}_{\rm inv}(Z;\calF)$ of 
 smoothing operators on a proper cocompact $\sfG$-manifold $Z$, where $\sfG$ is a {\em unimodular} Lie 
 groupoid. Let us first recall the definition of unimodularity: as before $\sfG\rightrightarrows M$ is a 
 Lie groupoid. Consider the real line bundle $L:=\left|\bigwedge^{\rm top}\right|T^*M\otimes\left|\bigwedge^{\rm 
 top}\right|\sfA$ over $M$. As shown in \cite{elw}, $L$ carries a canonical representation of $\sfG$.
 \begin{definition}
 \label{unimodular}
 A Lie groupoid $\sfG\rightrightarrows M$ is {\em unimodular} if there exists a nonvanishing, $\sfG
 $-invariant section of $L$.
 \end{definition}
 When $\sfG$ is unimodular, we refer to such an invariant section as a ``volume form'', or 
 ``transversal density'', usually denoted by $\Omega$. From now in this section we assume $\sfG$ 
 to be unimodular and choose such a volume form.
 
 \begin{remark}
Recall, cf.\ \cite{gs}, that for a smooth submersion $f:X\to Y$, integration over the fiber defines a map
\[
f_*:\Gamma^\infty_c(X,\left|\textstyle\bigwedge^{\rm top}T^*X\right|)\to \Gamma_c^\infty(Y,\left|\textstyle\bigwedge^{\rm top}T^*Y\right|),\quad 
\alpha\mapsto\int_f\alpha.
\]
More precisely, we choose an Ehresmann connection on the fibers of $f$, i.e., a smooth 
isomorphism $T_xX\cong T_x(f^{-1}(y))\oplus (f^*TY)_x$, for $f(x)=y$. With this, we can decompose $\left|\bigwedge^{top}T_x^*X\right|
\cong \left|\bigwedge^{\rm top} T_x^*(f^{-1}(y))\right|\otimes f^*(\left|\bigwedge^{\rm top} T^*Y\right|)$, and integrate over 
the fibers of $f$. One checks that the resulting density is independent of the Ehresmann 
connection chosen.
\end{remark}

Given $\Omega\in\Gamma^\infty(M;L)$, choose a cut-off density $c$ on $Z$ adapted to the $\sfG$-action, and define the following functional on $\Psi^{-\infty}_{\rm inv}
(Z;\calF)$:
\begin{equation}
\label{def-trace}
\begin{split}
\tau_\Omega(K):=&\int_Z c(z) k(z,z)d\nu\wedge \mu^*\Omega\\
=&\int_M\left(\int_\mu c(z) k(z,z)d\nu\right)\Omega,\quad K\in \Psi^{-\infty}_{\rm inv} 
(Z;\calF).
\end{split}
\end{equation}
In this formula, we have used the duality between $\sfA$ and $\sfA^*$ to pair $c$ with $
\mu^*\Omega$. The resulting ``transverse density'' combines with the fiberwise density $d\nu$ to 
form a density on $Z$,  using  the choice of an Ehresmann connection on fibration defined by  the moment map $\mu: Z\to M$. The resulting integral over $Z$ is independent of this choice by the remark above. 

In the rest of this article, we will include $d\nu$ together with the Schwarz kernel function $k(z_1, z_2)$, and view $k(z_1, z_2)d\nu(z_2)$ as a fiberwise density on $Z\fibprod{\mu}{\mu}Z$ along the second $\mu$-fiber. By slight abuse of notation, we will denote it again by $k(z_1, z_2)$.    
\begin{proposition}
\label{trace_prop}
$\tau_\Omega$ does not depend on the choice of $c$ and defines a trace if $\Omega$ is $\sfG$-invariant.
\end{proposition}
\begin{proof}

To check the trace property, we write out
\begin{align*}
\tau_\Omega([K_1,K_2])&= \int_M\int_\mu\int_\mu c(z_1)(k_1(z_1,z_2)k_2(z_2,z_1)-k_2(z_1,z_2)k_1(z_2,z_1))\Omega\\
&=\int_M\int_{\sfG^{\mu(g)}}\left<\varphi(g),\Omega\right>,
\end{align*}
where $\varphi\in \Gamma^\infty(\sfG;s^*\left|\bigwedge^{\rm top}\sfA^*\right|)$ is defined as
\begin{align*}
\varphi(g):=\int_{\mu^{-1}(s(g))}\int_{\mu^{-1}(s(g))} c(z_1) c(g^{-1}z_2)\big(k_1(z_1,&z_2)k_2(z_2,z_1)\\
&-k_2(z_1,z_2)k_1(z_2,z_1)\big).
\end{align*}
Above, we have used the defining property \eqref{eq-cut-off} of Def. \ref{cut-off} of the cut-off function to go from the first to the second line above. Next,  we have
\begin{align*}
\varphi(g^{-1})&=\int_{z_1\in\mu^{-1}(t(g))\atop z_2\in\mu^{-1}(t(g))} c(z_1) c(gz_2)\big(k_1(z_1,z_2)k_2(z_2,z_1)\\
&\hspace{6cm}-k_2(z_1,z_2)k_1(z_2,z_1)\big)\\
&=\int_{z_1\in\mu^{-1}(t(g))\atop z_2\in\mu^{-1}(s(g))} c(z_1) c(z_2)\big(k_1(z_1,g^{-1}z_2)k_2(g^{-1}z_2,z_1)\\
&\hspace{6cm}-k_2(z_1,g^{-1}z_2)k_1(g^{-1}z_2,z_1)\big)\\
&=\int_{z_1\in\mu^{-1}(t(g))\atop z_2\in\mu^{-1}(s(g))} c(z_1) c(z_2)\big(k_1(gz_1,z_2)k_2(z_2,gz_1)
\\
&\hspace{6cm}-k_2(gz_1,z_2)k_1(z_2,gz_1)\big)\\
&=\int_{z_1\in\mu^{-1}(t(g))\atop z_2\in\mu^{-1}(t(g))} c(g^{-1}z_1) c(z_2)\big(k_1(z_1,z_2)k_2(z_2,z_1)\\
&\hspace{6cm}-k_2(z_1,z_2)k_1(z_2,z_1)\big)\\
&=-\varphi(g).
\end{align*}
In this computation we have twice used the change of variables given by the diffeomorphism $g:
\mu^{-1}(s(g))\to\mu^{-1}(t(g))$ to go to the second and the fourth line. To go to the third line, the $\sfG$-invariance of $k_1$ and $k_2$ was used. Finally, because of this property $\varphi(g^{-1})=-
\varphi(g)$, the integral over $\sfG$ vanishes, proving the trace property.

To prove that the trace does not depend on the choice of cut-off function, let $c'$ be another such 
function, and insert $1= \int_{\sfG^{\mu(z)}}c'(g^{-1}z)$ into the formula for $\tau_\Omega$ and 
change variables again, using the $\sfG$-invariance of the kernels $k$:
\[
\begin{split}
\tau_\Omega(k)&=\int_M\int_\mu c(z) k(z,z)\\
&=\int_M\int_\mu \int_{\sfG^{\mu(z)}}c'(g^{-1}z)c(z) k(z,z)\\
&=\int_M\int_\mu\int_{\sfG_{\mu(z)}}c'(z)c(gz)k(z,z)\\
&=\int_M\int_\mu c'(z)k(z,z).
\end{split}
\]
This shows that the trace does not depend on the choice of cut-off function $c$.
\end{proof}

\begin{proposition}
\label{trace-symbol}
For $a\in \sym^{-\infty}_{\rm inv}\left(Z;\calF \right)$:
\[
\tau_\Omega\left(\op(a)\right)=\int_{T^*_\mu Z} c(z) a(z,\xi)\mu^*\Omega.
\]
\end{proposition}
\begin{proof}
The proof follows immediately by writing out the kernel of $\Op(a)$ defined by \eqref{quantization}, 
and substituting it in the definition \eqref{def-trace} of the trace $\tau_\Omega$.
\end{proof}

\subsection{The pairing}
In \cite{conmos} the higher index of elliptic operators on a manifold was defined by a 
pairing  with Alexander--Spanier cocycles representing cohomology classes of the underlying 
manifold.  In our set up of a proper groupoid action, we need a similar representation of the 
foliated invariant cohomology classes in $H^\bullet_\calF(Z)^\sfG$. The complex 
$C^\bullet_{\rm   diff}(Z)^\sfG$, with the differential \eqref{diff-def}, as well as its localization 
$C^\bullet_{\rm   loc}(Z)^\sfG$ to the unit space $M$ offers the right object to introduce the pairing. 
 
 In this form, cohomology classes naturally pair with the so-called localized $K$-theory of the 
algebra of smoothing operators $\Psi^{-\infty}_{\rm inv}(Z;\calF)$. The localization of the $K$-theory 
of this algebra to the diagonal $M$ is defined just as in \cite{ppt}, which in turn is based on 
the construction in \cite{mw} for manifolds. This results in an  abelian group denoted by 
$K_0^{\rm loc}(\Psi^{-\infty}_{\rm inv}(\sfG;Z))$.  Let $\Psi^{-\infty}_{\rm inv}(Z;\calF)^+$ be the unitalization of $\Psi^{-\infty}_{\rm inv}(Z;\calF)$. Classes of $K_0^{\rm loc}(\Psi^{-\infty}_{\rm inv}(\sfG;Z))$ are represented by pairs 
\[
(P,e)\in \Mat_\infty\left(\Psi^{-\infty}_{\rm inv}(Z;\calF)^+\right)\times \Mat_\infty(\C))
\]
with support in some open neighbourhood $U\subset Z_\mu^{(2)}$ of $Z$,
that are idempotents, i.e., $P^2=P$, $e^2=e$ and satisfy
\[
P-e\in\Mat_\infty\left(\Psi^{-\infty}_{\rm inv}(Z;\calF)\right).
\]
Two such pairs $(P_0,e_0)$ and $(P_1,e_1)$ represent the same class if one can find a 
homotopy 
of idempotents $(P_t,e_t),~t\in[0,1]$ with support in $U$. With this, the localized $K$-theory 
$K_0^{\rm loc}(\Psi^{-\infty}_{\rm inv}(Z;\calF))$ is defined as the direct limit over open sets $U$ 
localizing to the diagonal.
By construction, the localized $K$-theory comes equipped with a canonical morphism 
\begin{equation}
\label{forgetful}
K_0^{\rm loc}(\Psi^{-\infty}_{\rm inv}(Z;\calF))\to K_0(\Psi^{-\infty}_{\rm inv}(Z;\calF))
\end{equation}
to the usual $K$-theory group by forgetting about the support of idempotents.

Next, given a cocycle $\varphi\in C^{2k}_{\rm  diff}(Z)^\sfG$, we define
 \begin{align}
 \label{pairing}
\notag \left<\varphi,(P,e)\right>:=
 \int_{Z^{(2k+1)}_\mu}c(z_0)\varphi(z_0&,\ldots,z_{2k})k_P(z_0,z_1)\cdots k_P(z_{2k},z_0)\mu_{(2k)}^*\Omega\\
 &-\int_{Z^{(2k+1)}_\mu}c(z_0)\varphi(z_0,\ldots,z_{2k})\tr(e)\mu_{(2k)}^*\Omega
 \end{align}
 In this formula, $k_P$ denotes the kernel of the matrix trace of $P$.
 \begin{remark}
 \label{rk-K}
 The pairing above can be written in terms of the natural pairing between cyclic homology and 
 cohomology of the algebra $\Psi^{-\infty}_{\rm inv}(Z;\calF)$;
 \[
  \left<\varphi,(P,e)\right>=\left<\chi_\Omega(\varphi),\Ch(P,e)\right>,
 \]
 where $\Ch:K_0(\Psi^{-\infty}_{\rm inv}(Z;\calF))\to HC_{\rm ev}(\Psi^{-\infty}_{\rm inv}(Z;\calF))$ is 
 the noncommutative Chern character to cyclic homology and $\chi_\Omega(\varphi)$ is a natural 
 cyclic cocycle defined using the trace $\tau_\Omega$, analogous to \cite[\S 1.3]{ppt}. 
 \end{remark}
  \begin{proposition}
 \label{index-pairing}
 Combined with the map \eqref{eq:alpha} and the isomorphism of Corollary \ref{cor:gpdcoh}, equation \eqref{pairing} defines a pairing
 \[
 \left<~,~\right>:H^{\rm ev}_{\rm diff}(\sfG)\times K_0(\Psi^{-\infty}_{\rm inv}(Z;\calF))\to\C.
 \]
 \end{proposition}
 \begin{proof}
 For $k_0,\ldots,k_{2n}\in\Psi^{-\infty}_{\rm inv}(Z;\calF)$ and $\varphi\in C^{2n}_{\rm  diff}(Z)^\sfG$,
 define the pairing
 \begin{align*}
 \left<\varphi,k_0\otimes\ldots\otimes k_{2n}\right>:=\int_{Z^{(2n+1)}_
 \mu}c(z_0)\varphi(z_0,\ldots,z_{2n}&)k_1(z_0,z_1)\cdots\\
 &\cdots k_{2n}(z_{2n},z_0)\mu_{(2n+1)}^*\Omega.
 \end{align*}
 Since the differential \eqref{diff-def} of the complex $C^{\bullet}_{\rm  diff}(Z)^\sfG$ can just be 
 identified with the Alexander--Spanier differential along the fibers of $\mu$, one proves, as in 
 \cite[Lemma 2.1.]{conmos}, by a straightforward computation that 
 \[
  \left<d\varphi,k_0\otimes\ldots\otimes k_{2n}\right>= \left<\varphi,b(k_0\otimes\ldots\otimes k_{2n})\right>,
 \]
 where $b$ is the Hochschild boundary. Since in the formula \eqref{pairing} of the pairing, the 
 kernels $k_P$ are idempotents in a matrix algebra over $\Psi^{-\infty}_{\rm inv}(Z;\calF)$, the result 
 now follows.
 \end{proof}
 Next, we localize the pairing: to the unit space $M$. Since 
 $K_0^{\rm loc}(\Psi^{-\infty}_{\rm inv}(Z;\calF))$ is by definition an inverse limit of 
 $K$-theory groups with support in a neighbourhood of $M$, and $C^{\bullet}_{\rm  loc}(Z)^\sfG$ is 
 the direct limit, we have:
 \begin{corollary}
 The pairing \eqref{pairing} localizes to a pairing 
 \[
\left<\ ,\ \right>_{\rm loc}:\  H^{\rm ev}_\calF(Z)^\sfG\times K_0^{\rm loc}(\Psi^{-\infty}_{\rm inv}(Z;\calF))\to\C,
 \]
 compatible with the forgetful map \eqref{forgetful} and the van Est map.
 \end{corollary}
 \begin{proof}
 Since 
 $K_0^{\rm loc}(\Psi^{-\infty}_{\rm inv}(Z;\calF))$ is by definition a direct limit of 
 $K$-theory groups with support in a neighbourhood of $M$, and $C^{\bullet}_{\rm  loc}(Z)^\sfG$ is 
 the projective limit of groups over such neighbourhoods, the pairing indeed localizes. 
 By the results of Section \ref{vanest}, the cohomology of $C^{\bullet}_{\rm  loc}(Z)^\sfG$ is equal to $H^\bullet_\calF(Z)^\sfG$, and the result follows.
 \end{proof}
 \begin{remark}\label{rmk:nonunimodular} We have focused on the construction of the (localized) 
 index pairing in the case of a unimodular Lie groupoid $\sfG$. The pairing can be extended to 
 the case of  general, i.e., nonunimodular, Lie groupoids by replacing $H^{\rm ev}_{\rm diff}(G)$ 
 and $H^{\rm ev}_{\calF}(Z)^\sfG$ by $H^{\rm ev}_{\rm diff}(G; L)$ and $H^{\rm ev}_{\calF}(Z; 
 \mu^*L)^\sfG$, i.e. 
 \begin{equation}
 \label{tw-pairing}
\left<\ ,\ \right>_{\rm loc}: H^{\rm ev}_\calF(Z; \mu^*L)^\sfG\times K_0^{\rm loc}(\Psi^{-\infty}_{\rm inv}(Z;
\calF))\to\C.
 \end{equation}
 Indeed the invariant volume form $\Omega$ defines a class in $H^0_{\rm diff}(\sfG;L)$ and using the 
 $H^\bullet_{\rm diff}(G)$-module structure on $H^\bullet_{\rm diff}(G;L)$, c.f.\ \cite{crainic}, one can view the product $\varphi\cdot\Omega$ in \eqref{pairing} as an element in $H^{2k}_{\rm diff}(\sfG;L)$. With this, one easily observes that the same formula defines a pairing as above for general elements in $H^{\rm ev}_{\rm diff}(\sfG;L)$, and, respectively $H^{\rm ev}_{\calF}(Z; 
 \mu^*L)^\sfG$.
  \end{remark}
 Finally, consider a $\sfG$-invariant family of elliptic differential operators along the fibers of 
 $\mu$. Using the $\sfG$-invariant pseudodifferential calculus of section \ref{sec_pseudo}, one constructs the 
 index class $\Ind(D)\in K_0(\Psi^{-\infty}_{\rm inv}(Z;\calF))$ in the standard way. In fact, we can 
 localize the support of this $K$-theory class in an arbitrary small neighbourhood of $Z\subset Z\fibprod{\mu}{\mu}Z$, so as
 to obtain, just as in \cite[Prop. 3.5.]{ppt} a {\em localized } index class $\Ind_{\rm loc}(D)\in K_0^{\rm loc}(\Psi^{-\infty}_{\rm inv}
 (Z;\calF))$.
 
 \section{The index theorem}\label{sec:index}
 In this section we prove the following index theorem:
 \begin{theorem}
 \label{index-thm}
 Let $\sfG$ be a Lie groupoid acting properly and cocompactly on a manifold $Z$.
Suppose that $D$ is an elliptic $\sfG$-invariant differential operator on $Z$, and $\alpha\in 
H^{2k}_{\calF}(Z;\mu^*L)^{\sfG}$. The index pairing \eqref{tw-pairing} evaluated on these elements is 
given by
 \[
 \left<\alpha,\Ind_{\rm loc}(D)\right>_{\rm loc}:=\frac{1}{(2\pi\sqrt{-1})^k}\int_{T^*_\mu Z}\pi^*\left<c,\alpha\right>\wedge \hat{A}(\calF^*)\wedge\ch(\sigma(D)).
 \]
 \end{theorem}
 \begin{remark}
 In this index formula, $L$ is the line bundle $\wedge^{\rm top}T^*M\otimes \wedge^{\rm top}A$ over M; $\hat{A}(\calF^*)$ is the $\hat{A}$-genus of the foliation $\calF^*$, as will be 
 introduced below, and $\ch(\sigma(D))$ is the foliated Chern character of the symbol of $D$.
 These characteristic classes live in $H^\bullet_{\calF^*}(T^*_\mu Z)^\sfG$. The term $\left<c,\alpha\right>$ means that we use the pairing between $\sfA$ and $\sfA^*$ to obtain a transverse density
 in $\mu^*\left|\textstyle\bigwedge^{\rm top}T^*M\right|$. After wedging with the classes in $H^\bullet_{\calF^*}(T^*_\mu Z)^\sfG$ we obtain a compactly supported differential form on $Z$ that can be integrated. 
 \end{remark}
As an immediate corollary we observe that if the foliated cohomology class lies in the image of the 
van Est map, the pairing only depends on the ``global'' index class in $K_0\left(\Psi^{-\infty}_{\rm inv}(Z;\calF)\right)$, so that we have:
\begin{corollary}
In the situation above, let $\nu\in H^{2k}_{\rm diff}(\sfG;L)$. Then we have
\[
 \left<\nu,\Ind(D)\right>:=\frac{1}{(2\pi\sqrt{-1})^k}\int_{T^*_\mu Z}\pi^*\left<c,\Phi_Z(\nu)\right>\wedge \hat{A}(\calF^*)\wedge\ch(\sigma(D)).
\]
\end{corollary}

 \subsection{The algebraic index theorem}
 Given the set-up as above, where the Lie groupoid $\sfG\rightrightarrows M$ acts on $Z$ with 
 moment map $\mu:Z\to M$, consider $P:=T^*_\mu Z$, the cotangent bundle along the fibers of 
 $\mu$. With the fiberwise cotangent bundle structure, it is easy to see that this manifold carries a 
 canonical regular Poisson structure $\Pi$ whose symplectic leaves are exactly the cotangent 
 bundles of the fibers of $\mu$, i.e., $\calF^*$, where $\calF=\ker (T\mu)$ defines the foliation of
  $Z$ by the fibers of $\mu$. The groupoid $\sfG$ naturally acts on $T^*_\mu Z$, and the Poisson structure is invariant for this action.
 
 The Fedosov construction is a well-known method to construct a formal deformation quantization
 of a symplectic manifold. It applies also to regular Poisson manifolds by a leafwise construction, 
 since the foliation by symplectic leaves is not singular. The Fedosov construction is differential 
 geometric in nature, it basically depends on the choice of a symplectic connection. Therefore,
 in our case where $\sfG$ acts on $T^*_\mu Z$ in a proper way and we can average with 
 respect to a cut-off 
 function, we obtain a $\sfG$-invariant leafwise symplectic connection on $T^*_\mu Z$, and it  yields a formal deformation quantization 
 classified by a foliated cohomology class of the form
 \[
\Omega:=\frac{1}{\hbar}\omega+\sum_{k=0}^\infty\hbar^k\Omega_k\in \frac{1}{\hbar}\omega+
H^2_{\calF^*}(T_\mu^*Z,\C[[\hbar]])^\sfG,
\]
 where   $\omega\in\Omega^2_{\calF^*}(T^*_\mu Z)^{\sfG}$ is the leafwise symplectic form. 
 We denote the resulting sheaf of noncommutative algebras by $ \mathscr{A}^\hbar_{T^*_\mu Z}$. 
 By construction, this sheaf carries an action of $\sfG$ by automorphisms.

For regular Poisson manifolds there is a very general ``algebraic index theorem'' for cyclic 
homology of a formal deformation first proved in \cite{nt-hol} in the context of symplectic Lie algebroids. For our purposes, we need a
 $\sfG$-equivariant version of this theorem, which we will briefly outline
using the alternative construction in \cite{pptold} of the so-called {\em cyclic trace density}: this is a 
morphism of complexes of sheaves
\begin{equation}
\label{trace-density}
\Psi:\!\left(\Tot_\bullet\left(\mathcal{BC}(\mathscr{A}^\hbar_{T^*_\mu Z})\right)\!, b+B\right)\to
 \left(\Tot_\bullet\left(\mathcal{B}\Omega_{\calF^*}\!\otimes\C[\hbar^{-1},\hbar]]\right)\!,d_{\calF^*}\right)
 \!.
\end{equation}
Here, $\mathcal{BC}(\mathscr{A}^\hbar_{T^*_\mu Z})$ denotes the sheafified Connes' $(b,B)$-complex 
computing cyclic homology, c.f.\ \cite{loday}, where $b$ denotes the Hochschild differential. On the other hand 
\[
\Tot_k\mathcal{B}\Omega_{\calF^*}:=\bigoplus_{i\geq 0}\Omega^{2r-2i-k}_{\calF^*},
\]
where $2r$ is the rank of $\calF^*$, and is equipped with the foliated de Rham differential 
$d_{\calF^*}$. The only two ingredients of the trace 
density \eqref{trace-density} are the choice of a symplectic connection as in the Fedosov 
construction, and a universal cyclic cocycle on the formal Weyl algebra of $\R^{2r}$, c.f.\ 
\cite{pptold}. Therefore, since in our set-up this connection is $\sfG$-invariant, one easily 
observes that the morphism \eqref{trace-density} is $\sfG$-equivariant for the natural action 
on the domain and range.

On the other hand there is the 
classical symbol map, combined with the Hochschild--Kostant--Rosenberg isomorphism 
given by
\[
\sigma(a_0\otimes\ldots\otimes a_k):=i^*(a_0(0)da_1(0)\wedge\ldots\wedge da_k(0))\in\Omega^k_{\calF^*}(T^*_\mu Z),
\]
where $a_j(0)\in\calC^\infty_{T^*_\mu Z}$ is the constant term in the $\hbar$ expansion of $a_j\in
\mathscr{A}^\hbar_{T^*_\mu Z}$, for $j=1,\ldots,k$, and $i:\calF^*\to T^*_\mu Z$ is the inclusion of the 
symplectic leaves. This morphism maps the Hochschild differential to zero whereas the $B$-differential is mapped to the foliated de Rham differential $d_{\calF^*}$. The algebraic index theorem for such deformation quantizations measures the 
discrepancy between these two maps. For this we need the $\hat{A}$-genus defined by 
\[
\hat{A}(\calF^*):=\prod_{i=1}^k\frac{ x_i\slash 2}{\sinh(x_i\slash 2)}\in H^{\rm ev}_{\calF^*}\left(T^*_\mu Z;\C\right)^\sfG,
\]
where the $x_i$ are the leafwise Chern roots with respect to an invariant almost complex structure 
compatible with the symplectic form. With this, the algebraic index theorem reads:
 \begin{theorem}
\label{alg-ind-reg-poisson}
The following diagram commutes after taking cohomology:
\[
  \xymatrix{\Tot_\bullet\left(\mathcal{BC}(\mathscr{A}^\hbar_{T^*_\mu Z})\right)\ar[rr]^{\sigma}\ar[drr]_{\Psi}&&
  \Tot_\bullet\left(\mathcal{B}\Omega_{\calF^*}\right) \ar[d]^{\cup\hat{A}(\calF^*)e^{-\Omega\slash 2\pi\sqrt{-1}\hbar}}
  \\
  &&
  \Tot_\bullet\left(\mathcal{B}\Omega_{\calF^*}\right)\otimes\C[\hbar^{-1}, \hbar]]}.
\]
Furthermore, all three morphisms are  equivariant for the $\sfG$-action, so that an invariant cyclic chain in $\Tot_\bullet\left(\mathcal{BC}(\mathscr{A}^\hbar_{T^*_\mu Z})\right)$ lands in the complex $\Tot_\bullet\left(\mathcal{B}\Omega_{\calF^*}\right)^{\sfG}\otimes\C[\hbar^{-1}, \hbar]]$.
\end{theorem}
Next, consider a formal difference $[e]-[f]$ of two idempotents $e=e_0+\hbar e_1+\ldots$,  and $f=f_0+\hbar f_1+\ldots$  in $\Mat_N(\mathscr{A}^\hbar_{T^*_\mu Z})$ representing a class in $K$-theory. The noncommutative Chern character, c.f.\ \cite{loday}, defines a map 
\[
\Ch:K_0\left(\mathscr{A}^\hbar_{T^*_\mu Z}\right)\to HC_\bullet \left(\mathscr{A}^\hbar_{T^*_\mu Z}\right).
\]
On the other hand, taking the zero order term $[e_0]-[f_0]$ defines an element in the foliated $K$-
theory $K^0_{\calF^*}(T^*_\mu Z)$. Recall, c.f. \cite{ms}, that foliated $K$-theory is the group 
completion of the semigroup of isomorphism classes of foliated vector bundles, and the ordinary 
commutative Chern 
character  in $K$-theory combines with the restriction map 
$i^*:H^\bullet(T^*_\mu Z)\to  H_{\calF^*}^\bullet(T^*_\mu Z)$ to define the foliated Chern character
 $\ch: K^0_{\calF^*}(T^*_\mu Z)\to H_{\calF^*}^\bullet(T^*_\mu Z)$. With this we have the following corollary 
that is used in the proof of the index theorem:
\begin{corollary}
\label{cor-chern}
For $e\in K_0\left(\mathscr{A}^\hbar_{T^*_\mu Z}\right)$, we the following equality in $H^\bullet_{\calF^*}(T^*_\mu Z)$ holds true:
\[
\Psi\left(\Ch(e)\right)=R\left(\hat{A}(\calF^*)\wedge\ch(\sigma(e))\wedge e^{-\Omega/2\pi\sqrt{-1}\hbar}\right).
\]
\end{corollary}
Here $R$ is the operator that multiplies the degree $2k$-part of an even differential form 
with $(2\pi\sqrt{-1})^{-k}$. The 
appearance of this factor is caused by the different normalizations of the noncommutative and 
commutative Chern characters.

\subsection{A lemma on $\sfG$-invariant cohomology}
The algebraic index theorem of the previous section yields an equality in the $\sfG$-invariant 
foliated cohomology $H^\bullet_{\calF^*}(T^*_\mu Z)^{\sfG}$. To obtain an actual number and 
relate to the higher index pairing of Proposition \ref{index-pairing}, we need an integration map.
Because of the noncompactness of $Z$, we need the cut-off function for this:
\begin{lemma}
\label{inv-int}
Let $\alpha\in\Omega_{\calF}^{\rm top}(Z)^\sfG$. The functional 
\[
\int_{Z\slash\sfG}\alpha:=\int_Z \alpha\left<c,\mu^*\Omega\right>
\]
vanishes on exact invariant forms and defines a linear map 
\[
H^{\rm top}_{\calF}(Z;\mathbb{K})^\sfG\to \mathbb{K},
\] 
for 
any field $\mathbb{K}$  containing $\mathbb{R}$, which does not depend on the choice of the cut-off function $c$.
\end{lemma}
\begin{proof}
Given $\beta\in\Omega^{\rm top-1}_{\calF}(Z)^\sfG$, we have to show that 
\[
\int_{Z\slash \sfG} d_\calF\beta=0.
\]
Let us first remark that differentiation of the identity \eqref{eq-cut-off} of Definition \ref{cut-off} yields the identity
\[
 \int_{\sfG^{\mu(z)}}g^*d_\calF c(z)=0,\quad\mbox{for all} ~z\in Z.
\]
Indeed, although strictly speaking $c\in\Gamma^\infty_c(Z,\mu^*\left|\bigwedge^{\rm top}\sfA^*\right|)$ is a density, since 
the density bundle is a pull-back along $\mu$ and $d_\calF$ is the differential along the fibers of $
\mu$, the above equation makes sense. With this we now compute:
\begin{align*}
\int_{Z\slash \sfG} d_\calF\beta&=\int_Z d_\calF\beta\left<c,\mu^*\Omega\right>\\
&=-\int_Z\beta \wedge d_\calF\left<c,\mu^*\Omega\right>&\mbox{(by Stokes' Theorem)}\\
&=-\int_Z\beta \wedge\left<d_\calF c,\mu^*\Omega\right>&\mbox{(since $d_\calF\mu^*\Omega=0$)}\\\
&=-\int_Z\int_{\sfG^{\mu}}g^*c\beta\wedge \left<d_\calF c,\mu^*\Omega\right>&\mbox{(c.f.\ Def \ref{cut-off} (ii))}\\
&=-\int_Z\int_{\sfG^{\mu}}c\beta \wedge \left<(g^{-1})^*d_\calF c(z),\mu^*\Omega\right>\\
&=0,
\end{align*}
where, to go to the fifth line, we have used that both $\beta$ and $\Omega$ are $\sfG$-invariant. This proves the first 
claim. The second claim, that the integration map does not depend on the choice of the cut-off 
function $c$, is proved in the same manner as in the proof of Proposition \ref{trace_prop}.
\end{proof}
\begin{remark}
There are two extensions of this Lemma:
\begin{itemize}
\item[$i)$]
The lift of the $\sfG$-action to $T^*_\mu Z$ is still proper, but it will not be cocompact. It is easy to 
see that the lift $\pi^*c$ is a cut-off density on $T^*_\mu Z$ adapted to the $\sfG$-action in the 
sense of Definition \ref{cut-off}. When one assumes that the differential forms have compact 
support along the $\calF^*$-direction, the same Lemma holds true for $T^*_\mu Z$ when one integrates with 
respect to $\pi^*c$.
\item[$ii)$] One can interpret $\alpha\mu^*\Omega$ as a foliated invariant differential form with 
values in the foliated flat line bundle $\mu^*L$. Clearly the proof of the Lemma holds true for any 
element in $\Omega^\bullet_{\calF}(Z;\mu^*L)^\sfG$, so that
\[
\int_Z\left<c,\alpha\right>,\quad\alpha\in\Omega^{\rm top}_\calF(Z;\mu^*L)^\sfG,
\]
vanishes on exact forms.
\end{itemize}
\end{remark}
\subsection{Proof of Theorem \ref{index-thm}}
With the results of the previous two subsections, the algebraic index theorem for $T^*_\mu Z$, and the 
cohomological nature of the integration map, the proof of Theorem \ref{index-thm} proceeds
in complete analogy with the proof of the index theorem in \cite{ppt}. Instead of providing full 
details, which in the end would repeat many similar statements of {\em loc. cit.}, we give a step by 
step outline, in which  each step can be easily proved by modifying the arguments of \cite{ppt}, to 
which we refer for further details.

First, assume $\sfG$ to be unimodular, and fix an invariant transverse density $\Omega$ as in 
Definition \ref{unimodular}.

\subsubsection*{Step 1: Asymptotic calculus.} Instead of the usual pseudodifferential calculus, one can go over to an asymptotic version
where symbols $a(\hbar,z,\xi)$ depend on an additional variable $\hbar\in [0,\infty)$ and have an 
asymptotic expansion near $\hbar\to 0$ of the form
\[
a\sim\sum_{k=0}^\infty \hbar^k a_{m-k}, \quad a_{m-k}\in\symb_\textup{inv}^{m-k} (Z;\calF).
\]
We write $\asymb^m_{\rm inv}(Z;\calF)$ for the space of such asymptotic symbols of order $m$, 
and $\jsymb^m_{\rm inv}(Z;\calF)$ for the subspace of symbols vanishing  at $\hbar=0$ up to all 
orders. Define the scaling operator $\iota_\hbar:\symb^\infty_{\rm inv}(Z;\calF)\to \symb^\infty_{\rm inv}(Z;
\calF)$ by $(\iota_\hbar a)(z,\xi):=a(z,\hbar\xi)$. With this we can define an associative product on $\asymb^\infty_{\rm inv}(Z;\calF):=\bigcup_m\asymb^m_{\rm inv}(Z;\calF)$ by
\begin{equation}
\label{def-dq}
a_1\circledast a_2:=\begin{cases} \sigma_\hbar\left(\op_\hbar(a_1)\circ\op_\hbar(a_2)\right),&\hbar>0,\\ a_1\cdot a_2,&\hbar=0,\end{cases}
\end{equation}
where $\op_\hbar:=\op\circ\iota_\hbar$ and $\sigma_\hbar:=\iota_{\hbar^{-1}}\circ\sigma$.
The quotient algebra $\mbA^\infty_{\rm inv}:=\asymb^\infty_{\rm inv}(Z;\calF)\slash\jsymb^\infty_{\rm inv}(Z;
\calF)$ is isomorphic to $\symb^\infty_{\rm inv}(Z;\calF)[[\hbar]]$ as a vector space, and the 
product above defines a deformation quantization $\mbA^\infty$ of $T^*_\mu Z$ 
which is $\sfG$-invariant. In the invariant subalgebra $\mbA^\infty_{\rm inv}$ there is the ideal $
\mbA^{-\infty}_{\rm inv}$ which supports a  $\C[\hbar^{-1},\hbar]]$-valued trace induced by the trace of Proposition \ref{trace-symbol}:
\[
 \tau_\Omega(a):=\frac{1}{\hbar^r}\int_{T^*_\mu Z}c(z) a(\hbar,z,\xi)\mu^*\Omega,
\]
where $r$ denotes the rank of the foliation $\calF$.
\subsubsection*{Step 2: Comparing traces.} By the classification of $\star$-products on regular Poisson 
manifolds, c.f.\ \cite{nt-hol}, the deformation quantization of $T^*_\mu Z$ defined by the asymptotic 
pseudodifferential calculus is isomorphic to a $\sfG$-invariant Fedosov quantization 
$\mathscr{A}^\hbar_{T^*_\mu Z}$. On the invariant part $\mathscr{A}^{\hbar,\sfG}_{T^*_\mu Z}$ , 
the cyclic trace density morphism \eqref{trace-density} defines a $\C[\hbar^{-1},\hbar]]$-valued trace
by restricting to the degree zero of the Hochschild complex
\[
\Psi^{2r}_{2r}:\calC_{0}(\mathscr{A}^{\hbar,\sfG}_{T^*_\mu Z})\to\Omega^{2r}_{\calF^*}\otimes
\C[\hbar^{-1},\hbar]],
\]
which satisfies $\Psi^{2r}_{2r}(a\star b-b\star a)=d_{\calF^*}$-exact. By Lemma \ref{inv-int}, 
it follows that  the functional
\[
\tr_\Omega(a):=\int_{T^*_\mu Z}\Psi^{2r}_{2r}(a)\left<c,\mu^*\Omega\right>
\]
defines a trace. Just like in \cite[Prop. 5.4]{ppt}, one proves that the two traces agree exactly: for $a
\in \mbA^{-\infty}_{\rm inv}$, we have $\tau_\Omega(a)=\tr_\Omega(a)$.
\subsubsection*{Step 3: Compatibility with cup-products.} As a next step, we compare the cyclic trace density 
with the trace density as follows: consider $\varphi\in C^k_{\rm diff}(\sfG;Z)$, and write it, for 
simplicity only, as $\varphi=\varphi_0\otimes\ldots\otimes\varphi_k$ with $\varphi_i$ depending on 
a single variable $z\in Z$. For $a=a_0\otimes\ldots\otimes a_k\in \mathscr{A}^{\hbar,\sfG}_{T^*_\mu Z}$, the formula
\[
\begin{split}
  \overline{\mathfrak{X}} (\varphi)(a) :=\tr_\Omega \left(a_0\star \pi^*\varphi_0\star\ldots
  \star a_k\star\pi^*\varphi_k \right)
  \end{split}
\]
defines a morphism of complexes 
\[
\overline{\mathfrak{X}}:\left(\hat{C}^k_{\rm diff}(\sfG;Z),d\right)\to \left(\Tot^
\bullet(\mathcal{BC}(\mathscr{A}^{\hbar,\sfG}_{T^*_\mu Z})),b+B\right).
\] 
The crucial identity is now:
\[
\overline{\mathfrak{X}}(\varphi)(a)=\int_{T^*_\mu Z}\Psi(a)\wedge \Phi_Z(\varphi)\left<c,\mu^*\Omega\right>.
\]
This equality corresponds to Prop. 5.8. of \cite{ppt}.

\subsubsection*{Step 4: Final computation.} In this last step we put all ingredients together to compute the 
index pairing \eqref{index-pairing}. We consider a cohomology class
 $\alpha\in H^{2k}_{\calF}(Z)^\sfG$ and represent it by a localized groupoid cocycle 
 $\varphi\in C^{2k}_{\rm diff}(\sfG;Z)$. Given a $\sfG$-invariant elliptic differential operator $D$ 
 along the fibers of $\mu:Z\to M$, its class 
 $\Ind_{\rm loc}(D)\in K_0^{\rm loc}\left(\Psi^{-\infty}_{\rm inv}(Z)\right)$
 is represented by an idempotent $k_D\in M_N(\Psi^{-\infty}_{\rm inv}(Z))$ (actually a formal 
 difference of idempotents, but this does not alter the computation) which we write as $
 \Op(a)$ with $a\in\mbA^{-\infty}_{\rm inv}$. Next we scale the operator $D$ by $\hbar\in[0,\infty)$ 
 by the rule $\partial/\partial z_i\mapsto \hbar^{-1}\partial/\partial z_i$ in local coordinates $z_i$ 
 along the fibers of $\mu$, and observe that the localized $K$-theory class represented by $\Op_
 \hbar(a)$ does not depend 
 on $\hbar$. The pairing with $\alpha$ can then be computed in the limit $\hbar\to 0$:
 \begin{align*}
&\left<\Ind_{\rm loc}(D),\alpha\right>_{\rm loc} \\
 & \hspace{0.3cm}=\tau_\Omega\left(\varphi_0k_D\varphi_1k_D\cdots\varphi_{2k}k_D\right)&\mbox{(c.f.\ Prop. \ref{index-pairing})}\\
  & \hspace{0.3cm}=\lim_{\hbar\to 0}\tau_\Omega\left(\varphi_0\Op_\hbar(a)\varphi_1\Op_\hbar(a)\cdots
  \varphi_{2k}\Op_\hbar(a)\right)\\
  & \hspace{0.3cm}= \lim_{\hbar\to 0}\tau_\Omega\left(\Op_\hbar(\varphi_0\star a)\Op_\hbar(\varphi_1\star a)\cdots
\Op_\hbar(  \varphi_{2k}\star a)\right)&\mbox{(c.f.\ \eqref{def-dq} \& \eqref{quantization})} \\
&\hspace{0.3cm}=\lim_{\hbar\to 0}\tau_\Omega\left(\Op_\hbar(\varphi_0\star a\star \varphi_1\star a\star \varphi_{2k}\star a)\right)&\mbox{(c.f. \eqref{def-dq})}\\
&\hspace{0.3cm}=\lim_{\hbar\to 0}\tr_\Omega\left(\varphi\star a\star \varphi_1\star a\star \varphi_{2k}\star a\right)&\mbox{(by Step 2)}\\
&\hspace{0.3cm}=\lim_{\hbar\to 0} \int_{T^*_\mu Z\slash\sfG}\Psi(a\otimes\ldots\otimes a)\wedge \Phi_Z(\varphi)&\mbox{(by Step 3)}\\
&\hspace{0.3cm}=\frac{1}{(2\pi\sqrt{-1})^k} \int_{T^*_\mu Z\slash\sfG}\hat{A}(\calF^*)\ch(\sigma(D))\pi^*\alpha. &\mbox{(by Cor.  \ref{cor-chern})}
 \end{align*}
The integral in the last two lines is the invariant integral defined in Lemma \ref{inv-int} using the 
auxiliary cut-off function $c$. In the final line we have also used the fact that $a\otimes\ldots
 \otimes a$ represents the components of the noncommutative Chern character of the localized index class, c.f.\ Remark \ref{rk-K}. The characteristic class of the formal quantization constructed 
 by the asymptotic pseudodifferential calculus in 1) above is given by the leafwise symplectic 
 form, and since its cohomology class is trivial, this term does not appear in the index formula.  
 This finishes the proof of Theorem \ref{index-thm} for proper actions of unimodular Lie groupoids.
\begin{remark}
In this final remark we explain how to prove the main theorem in the case where $\sfG$ is not 
unimodular. As remarked before, in this case, there is no trace on the algebra $\Psi^{-\infty}_{\rm 
inv}(Z;\calF)$ and one has to modify the pairing to a pairing with groupoid cocycles twisted by the 
representation $L$ as in \eqref{tw-pairing}. To compute this pairing, we follow the strategy in 
\cite[\S 6]{ppt}: we trivialize the real line bundle $L$ by choosing a nonvanishing volume form $
\Omega\in\Gamma^\infty(M;L)$.
Its failure to be $\sfG$-invariant is measure by the so-called modular function 
\[
\delta^\Omega_{\sfG}(g):=\frac{g^*\Omega}{\Omega},
\]
which satisfies 
$\delta^\Omega_\sfG(g_1g_2)=\delta^\Omega_{\sfG}(g_1)\delta^\Omega_\sfG(g_2)$.
After this trivialization of $L$, the complex $C^\bullet_{\rm diff}(Z;\mu^*L)$ identifies with 
$C^\infty_{\delta-\rm inv}(Z^{(\bullet)}_\mu)$, where $\delta$-inv means that the functions satisfy 
the twisted invariance property
\[
\varphi(gz_0,\ldots,gz_k)=\delta^\Omega_{\sfG}(g^{-1})\varphi(z_0,\ldots,z_k).
\]
The differential is still given by \eqref{diff-def}. If we examine the proof of Proposition \ref{trace_prop}, we see that this twisting exactly
 compensates the failure of $\tau_\Omega$ to be a trace, making the pairing well-defined.
 
 After localization, we find that the complex $\Omega_\calF(Z;\mu^*L)^\sfG$ computing 
 $H^\bullet_\calF(Z;\mu^*L)^\sfG$ can be identified with the invariant foliated de Rham complex 
 $\Omega_\calF(Z)^\sfG$, again with a twisted differential that makes the integral of Lemma 
 \ref{inv-int} vanishing on (twisted-)exact forms. With this observation, the argument applying
 the algebraic index theorem remains valid in this case and one finds exactly Theorem
  \ref{index-thm}.
\end{remark}

 \section{Examples}\label{sec:example}
 
\subsection{Principal bundles} 
As a particular example, let us consider the case where $\sfG$ acts freely on $Z$, i.e., $Z$ is a principal $\sfG$-bundle. For simplicity, we will assume that $\sfG$ is unimodular, and there is a $\sfG$-invariant volume element $\Omega$. In this case the quotient space 
$B:=Z\slash\sfG$ is a smooth manifold, assumed to be compact. For principal bundles, many of the 
constructions in this paper have a natural interpretation in terms of the so-called {\em gauge 
groupoid} $\sfG(Z)$ of $Z$: this is the groupoid over the base $B$ with the space of arrows defined 
by
\[
\sfG(Z):=(Z\times_M Z)\slash\sfG,
\]
and the groupoid structure is induced by the $\sfG$-equivariant groupoid structure of the pair 
groupoid $Z\times_M Z\rightrightarrows Z$ associated to the moment map $\mu:Z\to M$. In fact, 
the principal bundle $Z$ defines a Morita equivalence between $\sfG$ and $\sfG(Z)$. 

The Lie algebroid $\sfA(Z)$ of the gauge groupoid $\sfG(Z)$ is given defined on the vector bundle 
$T_\mu Z\slash\sfG\to B$. When dealing with foliated invariant differential forms on $T^*_\mu\sfG$, 
we expect in the case of a free action to be able to push everything down to $B$, or better to the dual 
of the Lie algebroid $\sfA^*(Z)$. The space $T^*_\mu Z$ is foliated itself by the fibers of the 
composition of the cotangent projection with $\mu$, abusively also denoted $\mu$, and with this 
we have
\[
T_\mu(T^*_\mu Z)\slash\sfG\cong \pi^!\sfA(Z),
\]
where $\pi:\sfA^*(Z)\to B$ is the projection. Here $\pi^!\sfA(Z)$ is the pull-back Lie algebroid as in 
\cite{mack}, and the isomorphism is proved just as in \cite[Lemma 4.3]{ppt}. Since we pulling back 
$\sfA(Z)$ along the projection of its dual, we observe, just as in {\em loc. cit.}, that there exists a canonical symplectic form $\Theta\in\Omega^2_{\pi^!\sfA(Z)}$. 

\begin{proposition}\label{prop:integration-B}
Let $\sfG$ act freely on $Z$, and take  $\alpha\in\Omega^{\rm top}_\calF(T^*_\mu Z)^\sfG$ with 
compact support along the fibers of $T^*_\mu Z\to Z$. Then we have 
the equality
\[
\int_{T^*_\mu Z\slash\sfG}\alpha=\int_{\sfA^*(Z)}\left<\alpha|_{\sfA^*(Z)},\Theta^r\right>,
\]
where $r$ is the rank of the Lie algebroid $\sfA(Z)$.
\end{proposition}
\begin{proof}
Let $q:T^*_\mu Z\to \sfA^*(Z)$ be the quotient map. We pull back $\Theta$ to an element $q^*\Theta\in \Omega^2_\calF(T^*_\mu Z)^\sfG$ on $T^*_\mu Z$. $q^*\Theta^r\wedge \mu^* \Omega$ defines a $\sfG$-invariant element of $\Omega^{\rm top}(T^*_\mu Z)\otimes \mu^*L$. The $\sfG$-invariance property makes $q^*\Theta^r\wedge \mu^* \Omega$ naturally descend to a volume element $\Omega_{\sfA^*(Z)}$ with respect to which the integral on $\sfA^*(Z)$ is well defined. As a $\sfG$-invariant element of $\Omega^{\rm top}_{\calF}(T^*_\mu Z)$, $\alpha$ is the pullback of $\alpha|_{\sfA^*(Z)}$. Then 
\begin{align}
\label{eq:integral-B}
\int_{T^*_\mu Z\slash\sfG}\alpha&=\int_{T^*_\mu Z}\left<\pi^*c(z), \alpha\wedge \mu^*\Omega\right>\\
\notag&=\int_{T^*_\mu Z}\left<\pi^*c(z), q^* \alpha|_{\sfA^*(Z)}\wedge \mu^*\Omega\right>.
\end{align}
Now $\alpha$ can be written as 
\[
\left<q^*\alpha|_{\sfA^*{Z}}, q^*\Theta\right>q^*\Theta=q^*(\left<\alpha|_{\sfA^*(Z)}, \Theta\right>\Theta).
\] 
Continuing the computation in Eq. (\ref{eq:integral-B}), we obtain that 
\[
\begin{split}
\int_{T^*_\mu Z\slash\sfG}\alpha&=\int_{T^*_\mu Z}q^*(\left<\alpha|_{\sfA^*(Z)}, \Theta\right>)\left<\pi^*c(z), q^*\Theta^r\wedge \mu^* \Omega\right>\\
&=\int_{\sfA^*(Z)}\left<\alpha|_{\sfA^*(Z)}, \Theta\right>\Omega_{\sfA^*(Z)}\int_{\sfG_{\mu(z)}}c(g^{-1}z).
\end{split}
\]
By the definition of the cut-off function $c$, we conclude the identity of the proposition. 
\end{proof}
We apply Proposition \ref{prop:integration-B} to the index formula in Theorem \ref{index-thm}, and 
obtain the following index formula for free actions:
\begin{theorem}\label{thm:free} Suppose that $Z$ is a principal $\sfG$-bundle. For a closed invariant form $\alpha\in \Omega^{2k}_\calF(Z)^G$, we have:
\[
 \left<\alpha,\Ind_{\rm loc}(D)\right>_{\rm loc}:=\frac{1}{(2\pi\sqrt{-1})^k}\int_{\sfA^*(Z)}\left<\alpha|_{\sfA^*(Z)}\wedge \hat{A}(\calF^*)\wedge\ch(\sigma(D)), \Theta\right> .
\]
\end{theorem}
When $M$ is compact, the left translation defines a proper and cocompact $\sfG$ action on $Z=\sfG$ 
with the quotient being $M$. The gauge groupoid $\sfG\times_M \sfG\slash\sfG$ is identical to $\sfG$. 
Theorem \ref{thm:free} in this case recovers the main theorem proved in \cite[Thm. 5.1]{ppt}. 
\begin{remark}
Let us finally explain how the index theorem above for principal $\sfG$-bundles is connected to the 
index theorems of Connes, c.f.\ \cite[Sec. III.$7.\gamma$]{connes}, and Gorokhovsky--Lott, c.f.\ 
\cite{G}. For this we assume $\sfG$ to be a foliation groupoid. This means, c.f.\ \cite{cm} that the anchor of its Lie algebroid $\rho:\sfA\to TM$ is injective, and its image defines the foliation. If we write $\nu_\sfA:=TM\slash \rho(\sfA)$ for the normal bundle to the foliation,
the line bundle $L$ needed as a 
twisting in the non-unimodular case can be identified with $L\cong \bigwedge^{\rm top}\nu_\sfA^*=O(\nu_\sfA)$, 
the orientation bundle of $A$. In \cite[Remark 7.4.]{ppt}, we have constructed a canonical map
\begin{equation}
\label{com}
i: H^\bullet_{\rm diff}(\sfG;L)\to H^\bullet(BG;O(\nu_\sfA)).
\end{equation}
Now, let $\pi:Z\to B$ be a principal $\sfG$ bundle. The gauge groupoid $\sfG(Z)$ is also a foliation 
groupoid with induced foliation given by $\calF=\mu^*\sfA\slash \sfG$, and  normal bundle $\nu_{\calF}$ 
satisfying $\pi^*\nu_{\calF}\cong\mu^*\nu_{\sfA}$. With these identifications, the characteristic classes 
in Theorem \ref{thm:free} can be identified as the usual foliated characteristic classes of the induced 
foliation $\calF$ as in \cite[Ch. 5]{ms}. The right hand side of this integral can then be viewed as the 
result of pairing this foliated cohomology class with a transverse current, which can be identified, 
following the argument in the proof of Theorem 4 of \cite{G}, as $\psi^*(i(\alpha))$, where 
$\psi:B\to BG$ is the map, unique up to homotopy classifying $Z$. With this observation, Theorem
\ref{thm:free} yields the following result for foliation groupoids:
\begin{theorem}
Let $\sfG$ be a foliation groupoid and $\pi:Z\to B$ a principal $\sfG$-bundle equipped with an elliptic $
\sfG$-invariant differential operator $D$. For $c\in H^{2k}_{\rm diff}(\sfG;L)$, we have
\[
\left<c,{\rm Ind}(D)\right>=\frac{1}{(2\pi\sqrt{-1})^k}\int_B\hat{A}(\calF^*)\wedge \ch_\calF(\sigma(D))\wedge \psi^*(i(c)),
\]
where $\calF$ denotes the induced foliation on $B$ and $\psi:B\to BG$ is the classifying map of the principal bundle.
\end{theorem}
This is exactly the index theorem of \cite[Sec. III.$7.\gamma$]{connes} and \cite{G}, for the transverse 
currents coming from classes in the smooth groupoid cohomology. In that light, it should be remarked 
that the map \eqref{com} is not an isomorphism, so  \cite[Sec. III.$7.\gamma$]{connes} and \cite{G} 
give more general index theorems for the smaller class of foliation groupoids.
\end{remark}

\subsection{Homogeneous spaces of Lie groups}
Let $G$ be a Lie group, and $H$ be a compact subgroup. We consider the special case that $\sfG$ is a unimodular Lie group $G$, and $Z$ is the homogeneous space $G/H$. The map $\mu:Z\to point$ is the trivial map, and the foliation $\calF$ is the whole manifold $Z$.  Let $\mathfrak{g}$ and $\mathfrak{h}$ be the Lie algebra of $G$ and $H$, and $\mathfrak{g}^*$ and $\mathfrak{h}^*$ be their dual. Then by translation, $\Omega^\bullet(G/H)^G$ isomorphic to $(\wedge^\bullet T^*_{[e]}G/H)^H$, where $[e]$ is the coset $eH$ in $G/H$, and $(\wedge^\bullet T^*_{[e]}G/H)^H$ is isomorphic to $\left(\wedge^\bullet (\mathfrak{g}/\mathfrak{h})^*\right)^H$. Under these isomorphisms, the de Rham differential on $\Omega^\bullet(G/H)^G$ is isomorphic to the Lie algebra cohomology differential on the relative Lie algebra cohomology complex $\left(\wedge^\bullet (\mathfrak{g}/\mathfrak{h})^*\right)^H$, and therefore $H^*(Z)^G$ is isomorphic to $H^\bullet(\mathfrak{g}, H; \C)$. 

Let $D$ be a $G$-invariant elliptic differential operator on $Z=G/H$. The principal symbol $\sigma(D)$ defines a $K$-theory element on $T^*G/H=T^*_\mu Z$. Theorem \ref{index-thm} states that for any $\alpha\in H^\bullet(G/H)^G$, 
\begin{align}
\label{eq:homog}
 \left<\alpha, \Ind_{\rm loc}(D)\right>_{\rm loc}=&\\
 \notag  \frac{1}{(2\pi\sqrt{-1})^k}&\int_{T^*G/H}\left<\pi^*c(z), \alpha\right>\wedge \hat{A}(T^*G/H)\wedge \ch(\sigma(D)) .
\end{align}
As is explained in \cite[Remark 6.15]{wang}, by the $G$-invariance, the above integral can be reduced to an integral on $\mathfrak{m}^*$ where $\mathfrak{m}$ is a complement of $\mathfrak{h}$ in $\mathfrak{g}$.  More concretely, $G$ is a principal $H$ bundle  over $G/H$. By choosing a connection on this bundle, Connes and Moscovici \cite{cmL2} introduced an $\hat{A}$ class $\hat{A}(\mathfrak{g},H)\in H^\bullet(\mathfrak{g},H;\C)$. The restriction of the symbol $\sigma(D)$ of a $G$-invariant elliptic operator to $T^*_{[e]}G/H$ defines a $H$-equivariant $K$-theory element. Connes and Moscovici \cite{cmL2} introduced a Chern character $\ch(\sigma(D))_{\mathfrak{m}^*}$ of $\sigma(D)$ in $H^\bullet(\mathfrak{g}, H;\C)$.  Under the isomorphism between $H^\bullet(G/H)$ and $H^\bullet(\mathfrak{g},H;\C)$, Wang explained in \cite[Remark 6.15]{wang}, $\hat{A}(T^*G/H)$ is reduced to $\hat{A}(\mathfrak{g},H)\in H^\bullet(\mathfrak{g},H;\C)$ and $\ch(\sigma(D))$ is reduced to $\ch(\sigma(D))_{\mathfrak{m}^*}\in H^\bullet(\mathfrak{g},H;\C)$.  And the index pairing can be written as 
\[
 \left<\alpha, \Ind_{\rm loc}(D)\right>_{\rm loc}:=\frac{1}{(2\pi\sqrt{-1})^k}\left<\alpha\wedge \hat{A}(\mathfrak{g},H)\wedge \ch(\sigma(D))_{\mathfrak{m}^*},[V]\right>,
 \]
 where $[V]$ is the fundamental class of $\mathfrak{m}^*$. When $\alpha$ is $1\in H^0(\mathfrak{g}, H;\C)$, this is exactly the $L^2$-index theorem proved by Connes and Moscovici \cite{cmL2}. Now, with our theorem, the freedom of choosing different $\alpha$ provides a powerful approach to understand the full class $ \hat{A}(\mathfrak{g},H)\wedge \ch(\sigma(D))_{\mathfrak{m}^*}\in H^{\rm even}(\mathfrak{g},H; \C)$. 
 \begin{remark}
 Let us finally drop the unimodularity assumption and assume that $H\subset G$ is maximal 
 compact. The index theorem is now phrased in terms of the natural pairing
 \[
 \left<~,~\right>:H^\bullet(\g,H;\C)\times H^\bullet(\g,H;L)\to\C,
 \]
 where $L=\bigwedge^{\rm top}\g$. As remarked by Connes and Moscovici, in the case of the
 homogenous space space $Z=G\slash H$, the index class $\Ind(D)$ can be constructed in 
 $K_0(\calC^\infty_c(G))$, and we can pair with group cocycles $\nu\in H^{2k}_{\rm diff}(G;L)$ 
 using the pairing of \cite{ppt}. Since $H$ is maximal compact in $G$, when $G$ is connected, the van Est morphism $\Phi_Z$ of this paper induces an isomorphism $H^\bullet_{\rm diff}(G;L)\cong H^\bullet(\g,H;L)$: this is the classical van Est theorem for Lie groups \cite{ve}. Our main result now implies:
 \begin{theorem}
 Let $H\subset G$ be a maximal compact subgroup of a Lie group, and $D$ a $G$-invariant 
 elliptic differential operator on $Z:=G\slash H$. For $\nu\in H^{2k}_{\rm diff}(G;L)$ we have:
 \[
 \left<\nu,\Ind(D)\right>=\frac{1}{(2\pi\sqrt{-1})^k}\left<\hat{A}(\mathfrak{g},H)\wedge 
 \ch(\sigma(D))_{\mathfrak{m}^*},\Phi_Z(\nu)\right>
\]
 \end{theorem}
\end{remark}

 \subsection{Families on orbifolds}
By an \'etale Lie groupoid, we mean a Lie groupoid whose source and target maps are local diffeomorphisms. Consider a proper \'etale groupoid $\sfG$ such that the quotient space $X:=M/\sfG$ is compact. A proper \'etale groupoid $\sfG$ is unimodular. And the quotient space $M/\sfG$ is an orbifold, i.e., a topological manifold that is locally diffeomorphic to the quotient of an euclidean space by a finite group action.  Consider a proper $\sfG$ action on $\mu:Z\to M$. Let $D$ be $\sfG$-invariant elliptic differential operator on $Z$. Theorem \ref{index-thm} states 
\begin{equation}\label{eq:orbfld}
 \left<\alpha,\Ind_{\rm loc}(D) \right>_{\rm loc}:=\frac{1}{(2\pi\sqrt{-1})^k}\int_{T^*_\mu Z}\pi^*\left<c,\alpha\right>\wedge \hat{A}(\calF)\wedge\ch(\sigma(D)), 
\end{equation}
for $\alpha\in H^\bullet_\calF(Z)^\sfG$.  We observe that quotients of leaves of $\calF$ by the $\sfG$-action defines a foliation $\calF_X$ on $M/\sfG$. And $D$ defines a leafwise elliptic differential operator $D_X$ on $X$ with respect to the foliation $\calF_X$. The index pairing (\ref{eq:orbfld}) computes the index numbers of the operator $D_X$. This is a longitudinal index theorem for the leafwise elliptic operator $D_X$ on the orbifold $X$ with the foliation $\calF_X$. 

Assume $\mu:Z\to M$ is a fiber bundle over $M$ with a compact closed fiber $F$. Let $D$ be a $
\sfG$-invariant elliptic differential  operator on $Z$. For $x\in M$, $D|_{\mu^{-1}(x)}$ is an elliptic 
differential operator on a compact closed manifold $\mu^{-1}(x)$, and therefore the kernel and 
cockerel of $D|_{\mu^{-1}(x)}$ are finite dimensional vector spaces. Varying $x\in M$, as $D$ is $
\sfG$-invariant, $\ker(D)$ and ${\rm coker}(D)$ together define a $\sfG$-equivariant $K$-theory $
\Ind^t(D)$ element on $M$.  The latter group, also known as the ``orbifold $K$-theory'' $K^0_{\rm 
orb}(X)$, is the group completion of the monoid formed by isomorphism classes of $\sfG$-equivariant vector bundles on $M$. In general, the Chern character maps $K^\bullet_{\rm 
orb}(X)$ to the cohomology of the so-called ``inertia orbifold'', but here we are only interested in
its localization at the units $\ch_{[1]}$ giving a $\sfG$-invariant differential 
form on $M$. The integration of $\ch_{[1]}(\Ind^t(D))$ over $M/\sfG$ can be identified directly with the 
index pairing in this paper
\[
\int_{M/\sfG}\ch_{[1]}(\Ind^t(D))=\left<1, \Ind(D)\right>,
\] 
where $1$ stands for the constant function with value 1 on $M$. As $\sfG$ is proper,  up to a scalar, 
$1$ is the only nontrivial cohomology class in $H^\bullet_{\rm diff}(\sfG)$. By the index formula (\ref{eq:orbfld}), we have
\begin{theorem}\label{prop:family} Let $\sfG\rightrightarrows M$ be a proper \'etale groupoid, and $Z$ a $\sfG$-equivariant fiber bundle over $M$ with compact closed fibers $F$. Then we have:
\[
\int_{M/\sfG} \ch_{[1]}(\Ind^t(D))=\int_{T^*_\mu Z}\pi^*\left<c, \hat{A}(\calF)\wedge\ch(\sigma(D))\right>.
\]
\end{theorem}
Theorem \ref{prop:family} is an example of a family index theorem on orbifolds. We remark that as it is 
only the integration of $\ch_{[1]}(\Ind^t(D))$ over $M/\sfG$, not on the full inertia orbifold, the pairing 
above is in general a rational number, not an integer.

 \end{document}